\documentclass[11pt]{amsart}
\usepackage[T1]{fontenc}

\newcommand{\A}{\mathbb{A}}

\newcommand{\EE}{\mathbb{E}}

\newcommand{\NN}{\mathbb{N}}

\newcommand{\ZZ}{\mathbb{Z}}

\newcommand{\reg}{\mathrm{reg}}

\newcommand{\cC}{\mathcal{C}}

\newcommand{\cN}{\mathcal{N}}
\newcommand{\cO}{\mathcal{O}}

\newcommand{\cU}{\mathcal{U}}

\newcommand{\fa}{\mathfrak{a}}

\newcommand{\fb}{\mathfrak{b}}
\newcommand{\fC}{\mathfrak{C}}

\newcommand{\fg}{\mathfrak{g}}

\newcommand{\fn}{\mathfrak{n}}

\newcommand{\fR}{\mathfrak{R}}

\newcommand{\fsl}{\mathfrak{sl}}

\newcommand{\ft}{\mathfrak{t}}

\newcommand{\fu}{\mathfrak{u}}

\newcommand{\fv}{\mathfrak{v}}
\newcommand{\fw}{\mathfrak{w}}
\newcommand{\fx}{\mathfrak{x}}
\newcommand{\fy}{\mathfrak{y}}

\newcommand{\msd}{\mathsf{d}}

\newcommand{\dact}{\boldsymbol{.}}
\newcommand{\lra}{\longrightarrow}

\DeclareMathOperator{\Ad}{Ad}
\DeclareMathOperator{\ad}{ad}

\DeclareMathOperator{\Aut}{Aut}

\DeclareMathOperator{\Char}{char}

\DeclareMathOperator{\Der}{Der}

\DeclareMathOperator{\GL}{GL}
\DeclareMathOperator{\Gr}{Gr}

\DeclareMathOperator{\height}{ht}

\DeclareMathOperator{\Hom}{Hom}

\DeclareMathOperator{\im}{im}
\DeclareMathOperator{\id}{id}

\DeclareMathOperator{\Irr}{Irr}

\DeclareMathOperator{\Lie}{Lie}
\DeclareMathOperator{\Mat}{Mat}

\DeclareMathOperator{\modd}{mod}

\DeclareMathOperator{\msupp}{msupp}

\DeclareMathOperator{\pr}{pr}

\DeclareMathOperator{\rk}{rk}

\DeclareMathOperator{\ssrk}{rk_{ss}}
\DeclareMathOperator{\SL}{SL}

\DeclareMathOperator{\Sp}{Sp}

\DeclareMathOperator{\supp}{supp}

\usepackage{pictex}
\usepackage{amscd}
\usepackage{latexsym}
\usepackage{amssymb}
\usepackage{eucal}
\usepackage{eufrak}
\usepackage{verbatim}

\numberwithin{equation}{section}

\newtheorem{Theorem}{Theorem}[section]
\newtheorem{Lemma}[Theorem]{Lemma}
\newtheorem{Corollary}[Theorem]{Corollary}
\newtheorem{Proposition}[Theorem]{Proposition}

\theoremstyle{Theorem}
\newtheorem{Thm}{Theorem}[subsection]
\newtheorem{Lem}[Thm]{Lemma}
\newtheorem{Prop}[Thm]{Proposition}
\newtheorem{Cor}[Thm]{Corollary}

\newtheorem*{thm*}{Theorem}
\newtheorem*{cor*}{Corollary}

\theoremstyle{remark}
\newtheorem*{Remark}{Remark}
\newtheorem*{Remarks}{Remarks}

\newtheorem*{Example}{Example}

\numberwithin{equation}{section}

\begin{document}

\title{Commuting varieties for nilpotent radicals}

\author[Rolf Farnsteiner]{Rolf Farnsteiner}
\address{Mathematisches Seminar, Christian-Albrechts-Universit\"at zu Kiel, Ludewig-Meyn-Str. 4, 24098 Kiel, Germany}
\email{rolf@math.uni-kiel.de}
\date{\today}

\begin{abstract} Let $U$ be the unipotent radical of a Borel subgroup of a connected reductive algebraic group $G$, which is defined over an algebraically closed field $k$. In this paper,
we extend work by Goodwin-R\"ohrle concerning the commuting variety of $\Lie(U)$ for $\Char(k)\!=\!0$ to fields, whose characteristic is good for $G$.  \end{abstract}

\maketitle

\section*{Introduction}
Let $G$ be a connected reductive algebraic group, defined over an algebraically closed field $k$. Given a Borel subgroup $B \subseteq G$ with unipotent radical $U$, in this paper we
investigate two closely related varieties associated with the Lie algebra $\fu:=\Lie(U)$: The commuting variety $\cC_2(\fu)$, given by
\[ \cC_2(\fu):=\{(x,y) \in \fu\!\times\!\fu \ ; \ [x,y]=0\},\]
and the variety
\[ \A(2,\fu) := \{ \fa \in \Gr_2(\fu) \ ; \ [\fa,\fa]=(0)\},\]
of two-dimensional abelian subalgebras of $\fu$, which is a closed subset of the Grassmannian $\Gr_2(\fu)$ of $2$-planes of $\fu$. 

For $\Char(k)\!=\!0$, the authors proved in \cite{GR} that $\cC_2(\fu)$ is equidimensional if and only if the adjoint action of $B$ on $\fu$ affords only finitely many orbits. Being built on methods
developed in \cite[\S2]{Pr} for $\Char(k)=0$, their arguments don't seem to readily generalize to fields of positive characteristic. In fact, most of Premet's paper \cite{Pr} is devoted to the technically 
more involved case pertaining to fields of positive characteristic.

The purpose of this note is to extend the main result of \cite{GR} by employing techniques that work in good characteristics. For arbitrary $G$, this comprises the cases $\Char(k)\!=\!0$ as well as 
$\Char(k)\!\ge\!7$. Letting $Z(G)$ and $\modd(B;\fu)$ denote the center of $G$ and the modality of $B$ on $\fu$, respectively, our main result reads as follows:

\bigskip

\begin{thm*}  Suppose that $\Char(k)$ is good for $G$. Then
\[ \dim \cC_2(\fu)=\dim B\!-\!\dim Z(G)\!+\!\modd(B;\fu).\]
Moreover, $\cC_2(\fu)$ is equidimensional if and only if $B$ acts on $\fu$ with finitely many orbits. \end{thm*}

\bigskip
\noindent
If $\modd(B;\fu)\!=\!0$, then, by a theorem of Hille-R\"ohrle \cite{HR}, the almost simple components of the derived group $(G,G)$ of $G$ are of type $(A_n)_{n\le 4}$, or $B_2$. As in \cite{Pr} and 
\cite{GR}, the irreducible components are parametrized by the so-called distinguished orbits. 

Our interest in $\cC_2(\fu)$ derives from recent work \cite{CF} on the variety $\EE(2,\fu)$ of $2$-dimensional elementary abelian $p$-subalgebras of $\fu$, which coincides with $\A(2,\fu)$ whenever 
$\Char(k)\!\ge\!h(G)$, the Coxeter number of $G$. 

\bigskip

\begin{cor*} Suppose that $\Char(k)$ is good for a reductive group $G$ of semisimple rank $\ssrk(G)\!\ge\!2$. Then the following statements hold:
\begin{enumerate}
\item $\dim \A(2,\fu)=\dim B\!-\!\dim Z(G)\!+\!\modd(B;\fu)\!-\!4$.
\item $\A(2,\fu)$ is equidimensional if and only if $\modd(B;\fu)=0$.
\item $\A(2,\fu)$ is irreducible if and only if every component of $(G,G)$ has type $A_1$ or $A_2$. \end{enumerate} \end{cor*}

\bigskip
\noindent
For the reader's convenience, we begin by collecting a number of subsidiary results in the first two sections, some of which are variants of results in the literature.
Throughout this paper, all vector spaces over $k$ are assumed to be finite-dimensional.

\bigskip
\noindent
{\bf Acknowledgment.} I would like to thank Simon Goodwin for several helpful remarks, for pointing out a mistake in an earlier version, and for bringing references \cite{Go06} and \cite{GMR} 
to my attention.

\bigskip

\section{Preliminaries}
Let $\fg$ be a finite-dimensional Lie algebra over $k$, $\Aut(\fg)$ be its automorphism group. The commuting variety $\cC_2(\fg)$ is a conical closed subset of $\fg\!\times\!\fg$. Given a variety $X$, we 
denote by $\Irr(X)$ the set of irreducible components of $X$. Thus, each $C \in \Irr(\cC_2(\fg))$ is a conical closed subset of the affine space $\fg\!\times\!\fg$.

Recall that the group $\GL_2(k)$ acts on the affine space $\fg\!\times\!\fg$ via
\[ \left(\begin{smallmatrix}\alpha & \beta \\ \gamma & \delta \end{smallmatrix}\right)\dact (x,y) := (\alpha x\!+\!\beta y, \gamma x\!+\!\delta y),\]
with $\cC_2(\fg)$ being a $\GL_2(k)$-stable subset. In particular, the group $k^\times:=k\!\smallsetminus\!\{0\}$ acts on $\cC_2(\fg)$ via
\[ \alpha \dact (x,y) := \left(\begin{smallmatrix} 1 & 0 \\ 0 & \alpha \end{smallmatrix}\right)\dact (x,y) = (x,\alpha y).\]
We denote the two surjective projection maps by
\[ \pr_i : \cC_2(\fg) \lra \fg  \ \ \ \ \ \ \ (i \in \{1,2\}).\] 
Given $x \in \fg$, we let  $C_\fg(x)$ be the centralizer of $x$ in $\fg$. Since
\[ \pr_1^{-1}(x)=\{x\}\!\times\!C_\fg(x)\]
for all $x \in \fg$, the surjection $\pr_1: \cC_2(\fg) \lra \fg$ is a linear fibration $(\cC_2(\fg),\pr_1)$ with total space $\cC_2(\fg)$ and base space $\fg$. For any (not necessarily closed) subvariety $X 
\subseteq \fg$, we denote by $\cC_2(\fg)|_X$ the subfibration given by $\pr_1 : \pr_1^{-1}(X) \lra X$.  

\bigskip

\begin{Lemma} \label{Pre1} Let $X \subseteq \fg$ be a subvariety. Suppose that $C \subseteq \cC_2(\fg)|_X$ is a $k^\times$-stable, closed subset. Then $\pr_1(C)$ is a closed subset of $X$. 
\end{Lemma}

\begin{proof} We consider the morphism
\[ \iota : X \lra \cC_2(\fg)|_X \ \ ; \ \ x \mapsto (x,0).\]
Given $x \in \pr_1(C)$, we find $y \in \fg$ such that $(x,y) \in C$. By assumption, the map
\[ f : k \lra \cC_2(\fg)|_X \ \ ; \ \ \alpha \mapsto (x,\alpha y)\]
is a morphism such that $f(k^\times) \subseteq C$. Hence
\[ (x,0) = f(0) \in f(\overline{k^\times}) \subseteq \overline{f(k^\times)} \subseteq C,\]
so that $x \in \iota^{-1}(C)$. As a result, $\pr_1(C)=\iota^{-1}(C)$ is closed in $X$. \end{proof}

\bigskip

\begin{Lemma} \label{Pre2} Let $C \in \Irr(\cC_2(\fg))$. Then the following statements hold:
\begin{enumerate}
\item $\GL_2(k)\dact C = C$. 
\item The set $\pr_i(C)$ is closed. \end{enumerate} \end{Lemma}

\begin{proof} (1) This well-known fact follows from $\GL_2(k)$ being connected.

(2) As $C$ is $\GL_2(k)$-stable, Lemma \ref{Pre1} ensures that $\pr_1(C)$ is closed. By the same token, the map $(x,y) \mapsto (y,x)$ stabilizes $C$, so that $\pr_2(C)$ is closed as well.\end{proof}

\bigskip
\noindent
We next compute the dimension of $\cC_2(\fg)$ in terms of a certain invariant, that will be seen to coincide with the modality of certain group actions in our cases of interest.

Given $n \in \NN_0$, lower semicontinuity of ranks ensures that 
\[ \fg_{(n)} := \{ x \in \fg \ ; \ \rk(\ad x)\!=\!n\}\]
is a (possibly empty) locally closed subspace of $\fg$. We put $\NN_0(\fg):=\{n \in \NN_0 \ ; \ \fg_{(n)} \ne \emptyset\}$ and define
\[ \modd(\fg) := \max_{n \in \NN_0(\fg)} \dim \fg_{(n)}\!-\!n.\]

\bigskip
\noindent
Our next result elaborates on \cite[(2.1)]{GR}.

\bigskip

\begin{Proposition} \label{Pre3}The following statements hold:
\begin{enumerate}
\item Let $n \in \NN_0(\fg)$.
\begin{enumerate}
\item $(\cC_2(\fg)|_{\fg_{(n)}},\pr_1)$ is a vector bundle of rank $\dim_k\fg\!-\!n$  over $\fg_{(n)}$. In particular, the morphism $\pr_1 : \cC_2(\fg)|_{\fg_{(n)}} \lra \fg_{(n)}$ is open.
\item If $X \in \Irr(\fg_{(n)})$, then $\overline{\pr_1^{-1}(X)} \subseteq \cC_2(\fg)$ is irreducible of dimension $\dim X\!+\!\dim_k\fg\!-\!n$. 
\end{enumerate}
\item We have $\dim \cC_2(\fg)=\dim_k\fg\!+\!\modd(\fg)$. 
\item If $C \in \Irr(\cC_2(\fg))$, then 
\[\dim C = \dim \pr_1(C)\!+\!\dim_k\fg\!-\!n_C,\] 
where $n_C := \max\{n \in \NN_0 \ ; \ \fg_{(n)}\cap\pr_1(C) \ne \emptyset\}$. 
\item Let $X \in \Irr(\fg_{(n)})$ be such that $\overline{\pr_1^{-1}(X)} \in \Irr(\cC_2(\fg))$. Then we have 
\[ C_\fg(x) \subseteq \overline{X} \subseteq \overline{\fg_{(n)}} \subseteq \bigsqcup_{m\le n} \fg_{(m)} \ \ \ \  \text{for all} \ x \in X.\]
\item If $n \in \NN_0$ is such that $\modd(\fg)=\dim\fg_{(n)}\!-\!n$, then $\overline{\pr_1^{-1}(X)} \in \Irr(\cC_2(\fg))$ for every $X \in \Irr(\fg_{(n)})$ such that $\dim X = \dim \fg_{(n)}$. \end{enumerate} 
\end{Proposition} 

\begin{proof} (1a) If $V,W$ are $k$-vector spaces and $\Hom_k(V,W)_{(n)} := \{ f \in \Hom_k(V,W) \ ; \ \rk(f)\!=\!n\}$, then the map
\[ \Hom_k(V,W)_{(n)} \lra \Gr_{\dim_kV-n}(V) \ \ ; \ \ f \mapsto \ker f\]
is a morphism. Consequently,
\[ C_\fg : \fg_{(n)} \lra \Gr_{\dim_k\fg-n}(\fg) \ \ ; \ \ x \mapsto C_\fg(x)\]
is a morphism as well and general theory implies that 
\[ E_{C_\fg} := \{(x,y) \in \fg_{(n)}\!\times\!\fg \ ; \ y \in C_\fg(x)\}\]
is a vector bundle of rank $\dim_k\fg\!-\!n$ over $\fg_{(n)}$, which coincides with $\cC_2(\fg)|_{\fg_{(n)}}$, see \cite[(VI.1.2)]{Sh}. 

(1b) Given an irreducible component $X \in \Irr(\fg_{(n)})$, we consider the subbundle $\cC_2(\fg)|_X = \cC_2(\fg)\cap(X\!\times\!\fg)$ together with its surjection $\pr_1 : \cC_2(\fg)|_X \lra X$. 

Let $C \in \Irr(\cC_2(\fg)|_X)$ be an irreducible component. Since $\cC_2(\fg)|_X$ is $k^\times$-stable, so is $C$. In view of Lemma \ref{Pre1}, we conclude that $\pr_1(C)$ is closed in $X$. It now 
follows from \cite[(1.5)]{Fa04} that the variety $\pr_1^{-1}(X)$ is irreducible. Hence its closure enjoys the same property. Consequently, 
\[ \pr_1 : \overline{\pr_1^{-1}(X)} \lra \overline{X} \]
is a dominant morphism of irreducible affine varieties such that $\dim \pr_1^{-1}(x)=\dim_k\ker(\ad x) = \dim_k\fg\!-\!n$ for every $x \in X$. Since $X$ is locally closed, it is an open subset of 
$\overline{X}$. The fiber dimension theorem thus yields
\[ \dim \overline{\pr_1^{-1}(X)} = \dim \overline{X}\!+\!\dim_k\fg\!-\!n = \dim X\!+\!\dim_k\fg\!-\!n,\]
as desired.

(2) We have
\[ (\ast) \ \ \ \ \ \ \ \ \ \cC_2(\fg) = \bigcup_{n\in \NN_0(\fg)} \bigcup_{X \in \Irr(\fg_{(n)})} \overline{\pr_1^{-1}(X)},\]
whence
\[ \dim \cC_2(\fg) = \max_{n\in \NN_0(\fg)}\max_{X \in \Irr(\fg_{(n)})} \dim X\!+\!\dim_k \fg\!-\!n =  \max_{n\in \NN_0(\fg)} \dim \fg_{(n)}\!+\!\dim_k\fg\!-\!n = \dim_k\fg\!+\!\modd(\fg),\]
as asserted. 

(3) In view of (1b) and ($\ast$), there are $n_C \in \NN_0$ and $X_C \in \Irr(\fg_{(n_C)})$ such that 
\[ C = \overline{\pr_1^{-1}(X_C)}.\]
Since $\pr_1$ is surjective, we have $X_C = \pr_1(\pr_1^{-1}(X_C))$. Consequently, $\pr_1(C) = \pr_1(\overline{\pr_1^{-1}(X_C)}) \subseteq \overline{X_C}$, while $X_C \subseteq \pr_1(C)$
in conjunction with Lemma \ref{Pre1} yields $\overline{X_C} \subseteq \pr_1(C)$. Thus, lower semicontinuity of the rank function yields
\[ \pr_1(C) \subseteq \overline{\fg_{(n_C)}} \subseteq \bigsqcup_{n \le n_C} \fg_{(n)},\]
so that $\max\{n \in \NN_0 \ ; \ \pr_1(C)\cap \fg_{(n)}\ne \emptyset\} \le n_C$. On the other hand, $\emptyset \ne X_C \subseteq \pr_1(C)\cap \fg_{(n_C)}$ implies $n_C \le \max\{n \in \NN_0 \ ; \ 
\pr_1(C)\cap \fg_{(n)}\ne \emptyset\}$. Hence we have equality and (1b) yields
\[ \dim C = \dim X_C\!+\!\dim_k\fg\!-\!n_C = \dim \overline{X_C}\!+\!\dim_k\fg\!-\!n_C = \dim \pr_1(C)\!+\!\dim_k\fg\!-\!n_C,\]
as desired. 

(4) Let $x \in X$. Then we have $\{x\}\!\times\!C_\fg(x) = \pr_1^{-1}(x) \subseteq \overline{\pr_1^{-1}(X)}$. By assumption, the latter set is $\GL_2(k)$-stable, so that in particular
$C_\fg(x)\!\times\!\{x\} \subseteq \overline{\pr_1^{-1}(X)}$. It follows that
\[ C_\fg(x) \subseteq \overline{X} \ \ \ \ \ \ \forall \ x \in X.\]
Since $\overline{X} \subseteq \overline{\fg_{(n)}} \subseteq \bigsqcup_{m \le n} \fg_{(m)}$, our assertion follows. 

(5) This follows from (1b) and (2). \end{proof}

\bigskip

\begin{Corollary} \label{Pre4}The following statements hold:
\begin{enumerate}
\item The subset $\overline{\pr_1^{-1}(\fg_{(\max \NN_0(\fg))})}$ is an irreducible component of $\cC_2(\fg)$ of dimension $2\dim_k\fg\!-\!\max \NN_0(\fg)$. 
\item Suppose that $\cC_2(\fg)$ is equidimensional. Then we have $\modd(\fg)=\dim_k\fg\!-\!\max\NN_0(\fg)$.
\item Suppose that $\cC_2(\fg)$ is irreducible. Then we have $\dim\fg_{(n)}\!-\!n=\modd(\fg)$ if and only if $n=\max\NN_0(\fg)$. \end{enumerate} \end{Corollary} 

\begin{proof} (1) Let $n_0:= \max \NN_0(\fg)$. By lower semicontinuity of the function $x \mapsto \rk(\ad x)$,  $\fg_{(n_0)}$ is an open, and hence irreducible and dense subset of $\fg$.
Hence $\pr_1^{-1}(\fg_{(n_0)})$ is open in $\cC_2(\fg)$, and Proposition \ref{Pre3} shows that $C_{(n_0)}:= \overline{\pr_1^{-1}(\fg_{(n_0)})}$ is irreducible of dimension
$\dim \fg_{(n_0)}\!+\!\dim_k\fg\!-\!n_0 = 2\dim_k\fg\!-\!n_0$. Let $C \in \Irr(\cC_2(\fg))$ be such that $C_{n_0} \subseteq C$. Then $\pr_1^{-1}(\fg_{(n_0)})$ is a non-empty open subset of $C$,
so that $C_{n_0}= C \in \Irr(\cC_2(\fg))$. 

(2) This follows directly from (1) and Proposition \ref{Pre3}(2).

(3) Suppose that $n \in \NN_0(\fg)$ is such that $\modd(\fg)=\dim\fg_{(n)}\!-\!n$. Let $X \in \Irr(\fg_{(n)})$ be an irreducible component such that $\dim X = \dim\fg_{(n)}$. Thanks to Proposition 
\ref{Pre3}(5), $C_X:=\overline{\pr_1^{-1}(X)}$ is an irreducible component of $\cC_2(\fg)$, so that $C_X = \cC_2(\fg)$. Consequently,
\[ \fg = \pr_1(\cC_2(\fg)) = \pr_1(C_X) \subseteq \overline{X} \subseteq \bigcup_{m\le n} \fg_{(m)},\]
so that $\max\NN_0(\fg)\le n$. Hence we have equality. \end{proof} 

\bigskip
\noindent
In general, the value of $\modd(\fg)$ is hard to compute. For certain Lie algebras of algebraic groups and for those having suitable filtrations, the situation is somewhat better. 

\bigskip

\begin{Example} Let $\Char(k)=p\ge 5$ and consider the $p$-dimensional Witt algebra $W(1):=\Der_k(k[X]/(X^p))$, see \cite[(IV.2)]{SF} for more details. This simple Lie algebra affords a canonical 
descending filtration
\[ W(1)=W(1)_{-1} \supseteq W(1)_0 \supseteq \cdots \supseteq W(1)_{p-2} \supseteq (0),\]
where $\dim_kW(1)_i = p\!-\!1\!-\!i$. By way of illustration, we shall verify the following statements:
\begin{enumerate}
\item The variety $\cC_2(W(1))$ has dimension $p\!+\!1$ and is not equidimensional, with 
\[\Irr(\cC_2(W(1))) = \{\overline{\pr_1^{-1}(W(1)_{(\ell)})} \ ; \ \frac{p\!+\!1}{2}\! \le \! \ell \! \le p\!-\!1\}.\]
\item Let $\fb:=W(1)_0$. The variety $\cC_2(\fb)$ has pure dimension $p$, with 
\[\Irr(\cC_2(\fb)) = \{\overline{\pr_1^{-1}(\fb_{(\ell)})} \ ; \  \frac{p\!-\!1}{2}\!\le \!\ell\!\le \! p\!-\!2\}.\]
\item (cf.\ \cite[(4.3)]{YC}) Let $\fu:= W(1)_1$. The variety $\cC_2(\fu)$ has pure dimension $p$, with 
\[\Irr(\cC_2(\fu)) = \{\overline{\pr_1^{-1}(\fu_{(\ell)})} \ ; \ \frac{p\!-\!3}{2}\le\!\ell\!\le\!p\!-\!4\}.\]
\item (cf.\ \cite[(3.6)]{YC}) Let $\cN:=\{x \in W(1) \ ; \ (\ad x)^p=0\}$ be the $p$-nilpotent cone of $W(1)$. The variety $\cC_2(\cN):= \cC_2(W(1))\cap(\cN\!\times\!\cN)$ has pure dimension $p$ 
with 
\[\Irr(\cC_2(\cN))=\{\overline{\pr_1^{-1}(W(1)_{(\ell)})} \ ; \ \ell \in \{\frac{p\!+\!1}{2},\ldots,\!p\!-\!2\}\}\cup \{\overline{\pr_1^{-1}(W(1)_{(p-1)}\cap\cN)}\}.\] \end{enumerate}
\begin{proof} (1) Let $x \in W(1)\!\smallsetminus\!\{0\}$ and consider the Jordan-Chevalley-Seligman decomposition $x=x_s\!+\!x_n$, with $x_s$ semisimple, $x_n$ $p$-nilpotent and $[x_s,x_n] =0$,
(cf.\ \cite[(II.3.5)]{SF}). Since every maximal torus $\ft \subseteq W(1)$ is one-dimensional and self-centralizing, the assumption $x_s \ne 0$ entails $x_n \in C_{W(1)}(x_s)=kx_s$, so that $x_n=0$. 
As a result, every $x \in W(1)\!\smallsetminus\!\{0\}$ is either $p$-nilpotent or semisimple, and \cite[(2.3)]{YC} implies
\[ \ker (\ad x) = \left\{ \begin{array}{ccc}  W(1)_{p-1-i} & x \in W(1)_i\!\smallsetminus\! W(1)_{i+1} & \frac{p-1}{2}\!\le\!i\!\le\! p\!-\!2 \\ kx\!\oplus\! W(1)_{p-1-i} & x \in W(1)_{i}\!\smallsetminus\!W(1)_{i+1} 
                                                                   & 1\!\le\!i\!\le\!\frac{p-3}{2} \\
                                                                   kx & x \in W(1)\!\smallsetminus\!W(1)_1. & \end{array} \right.\]
This in turn yields
\[ W(1)_{(\ell)} = \left\{ \begin{array}{cc}  W(1)_{p-\ell}\!\smallsetminus\! W(1)_{p-\ell+1} & 2\!\le\!\ell\!\le\! \frac{p-1}{2} \\ 
 W(1)_{\frac{p-3}{2}}\!\smallsetminus\!W(1)_{\frac{p+1}{2}}  & \ell\!=\!\frac{p+1}{2} \\ W(1)_{p-\ell-1}\!\smallsetminus\!W(1)_{p-\ell} 
                                                                   & \frac{p+3}{2}\!\le\!\ell\le\!p\!-\!2 \\
                                                                   W(1)\!\smallsetminus\!W(1)_1 & \ell\!=\!p\!-\!1\\
                                                                   \{0\} & \ell\!=\!0\\
                                                                        \emptyset & \text{else.}  \end{array} \right.\]
We thus have $\modd(W(1))\!=\!1$, so that $\dim \cC_2(W(1))\!=\!p\!+\!1$. Moreover, each of the varieties $W(1)_{(\ell)}$ is irreducible, with $\overline{W(1)_{(\ell)}} = W(1)_{p-\ell}$ for
$2\!\le\!\ell\!\le\!\frac{p-1}{2}$. Proposition \ref{Pre3}(4) in conjunction with the above now shows that $\overline{\pr_1^{-1}(W(1)_{(\ell)})} \not \in \Irr(\cC_2(W(1)))$ for $2\!\le\!\ell\!\le\!\frac{p-1}{2}$.
Consequently,
\[ (\ast) \ \ \ \ \ \ \ \ \cC_2(W(1)) = \bigcup_{\frac{p+1}{2}\le \ell\le p-1} \overline{\pr_1^{-1}(W(1)_{(\ell)})}.\]
According to Corollary \ref{Pre4},
\[\overline{\pr_1^{-1}(W(1)_{(p-1)})} = \overline{\bigcup_{x \in W(1)\smallsetminus W(1)_1}\{x\}\!\times\!kx} \subseteq \{(x,y) \in \cC_2(W(1)) \ ; \ \dim_k kx\!+\!ky\!\le\!1\}\]
is an irreducible component of dimension $p\!+\!1$. Let $\ell \in \{\frac{p+1}{2}, \ldots, p\!-\!2\}$. Given $x \in W(1)_{(\ell)}$, it thus follows that
\[ \{x\}\!\times\!C_{W(1)}(x) \subseteq \overline{\pr_1^{-1}(W(1)_{(\ell)})} \ \ \text{while} \ \  \{x\}\!\times\!C_{W(1)}(x) \not \subseteq \overline{\pr_1^{-1}(W(1)_{(p-1)})},\]
whence 
\[ \overline{\pr_1^{-1}(W(1)_{(\ell)})} \not \subseteq \overline{\pr_1^{-1}(W(1)_{(p-1)})}.\]
Thanks to Proposition 1.3(3) we have 
\[\dim \overline{\pr_1^{-1}(W(1)_{(\ell)})} = \dim_kW(1)_{p-\ell-1}\!+\!\dim_kW(1)\!-\!\ell =p,\] 
so that there are no containments among the irreducible sets $(\overline{\pr_1^{-1}(W(1)_{(\ell)})})_{\frac{p+1}{2}\le \ell \le p-2}$. As a result, ($\ast$) is the decomposition of $\cC_2(W(1))$ into its 
irreducible components.
 
(2) We now consider the ``Borel subalgebra'' $\fb:= W(1)_0$ of dimension $p\!-\!1$. Writing $W(1)=ke_{-1}\!\oplus\!\fb$ with $C_{W(1)}(e_{-1})=ke_{-1}$, we have $(\ad x)(W(1))=
k[x,e_{-1}]\!\oplus\!(\ad x)(\fb)$ for all $x \in \fb$, whence  $\fb_{(\ell)} = W(1)_{(\ell+1)}$ for $1\!\le\!\ell\!\le\!p\!-\!3$, while $\fb_{(p-2)} = \fb\!\smallsetminus\!W(1)_1$. Consequently,
 \[ \dim \fb_{(\ell)} = \left\{ \begin{array}{cc}  \ell & 1\!\le\!\ell\!\le\! \frac{p-3}{2} \\ \ell\!+\!1 &  \frac{p-1}{2}\!\le\!\ell\le\!p\!-\!2\\ 0 & \ell\!=\!0\\ -1 & \text{else,} \end{array} \right.\]
where we put $\dim\emptyset = -1$. Thus, $\modd(\fb)=1$ and $\dim \cC_2(\fb)=p$. The arguments above show that $\overline{\pr_1^{-1}(\fb_{(\ell)})} \not \in \Irr(\cC_2(\fb))$, whenever 
$1\!\le\!\ell\!\le\! \frac{p-3}{2}$. In view of the irreducibility of $\fb_{(\ell)}$, Proposition \ref{Pre3}(5) shows that $\overline{\pr_1^{-1}(\fb_{(\ell)})}$ is an irreducible component of dimension $p$ for 
$\ell \in \{\frac{p-1}{2},\ldots, p\!-\!2\}$. 

(3) We next consider $\fu:=W(1)_1$ and observe that $\fu_{(\ell)} = \fb_{(\ell+1)}\cap\fu$ for $0\!\le\!\ell\!\le\!p\!-\!3$. Consequently,  
 \[ \dim \fu_{(\ell)} = \left\{ \begin{array}{cc}  \ell\!+\!1 & 0\!\le\!\ell\!\le\! \frac{p-5}{2} \\ \ell\!+\!2 &  \frac{p-3}{2}\!\le\!\ell\le\!p\!-\!4\\ -1 & \text{else,} \end{array} \right.\]
so that $\modd(\fu)=2$ and $\dim \cC_2(\fu)=p$. The remaining assertions follow as in (2). 

(4) In view of \cite[(2.3)]{YC}, we have $C_\fg(x) \subseteq \cN$ for all $x \in \cN\!\smallsetminus\!\{0\}$. This implies
\[\cC_2(\cN) = \bigcup_{2 \le \ell \le p\!-\!1} \overline{\pr_1^{-1}(W(1)_{(\ell)}\cap \cN)} = \bigcup_{2 \le \ell \le p\!-\!2} \overline{\pr_1^{-1}(W(1)_{(\ell)})}\cup\overline{\pr_1^{-1}(W(1)_{(p-1)}\cap\cN)}.\]
By the arguments above, we have $\pr_1^{-1}(W(1)_{(\ell))} \subseteq \bigcup_{\frac{p+1}{2}\le n \le p-2}\overline{\pr_1^{-1}(W(1)_{(n)})}$ for $\ell \in \{1,\ldots,\frac{p\!-\!1}{2}\}$, so that
\[ \cC_2(\cN) = \bigcup_{\frac{p+1}{2}\le \ell\le p-2} \overline{\pr_1^{-1}(W(1)_{(\ell)})}\cup (\overline{\pr_1^{-1}(W(1)_{(p-1)}\cap \cN)}.\]
By work of Premet \cite{Pr0}, the variety $\cN$ is irreducible of dimension $\dim \cN=p\!-\!1$. It follows that the dense open subset $W(1)_{(p-1)}\cap\cN$ is irreducible as well. Lemma
\ref{Pre1} implies that $\pr_1(C)$ is closed in $W(1)_{(p-1)}\cap\cN$ for every $C \in \Irr(\cC_2(\cN)|_{W(1)_{(p-1)}\cap\cN})$. Using \cite[(1.5)]{Fa04}, we conclude that
the variety 
\[  \pr_1^{-1}(W(1)_{(p-1)}\cap \cN)= \cC_2(\cN)|_{W(1)_{(p-1)}\cap\cN}\] 
is irreducible of dimension $p$. \end{proof} \end{Example}

\bigskip

\begin{Remarks} (1) \ In \cite[(Thm.5)]{Le} P.\ Levy has shown that commuting varieties of Lie algebras of reductive algebraic groups are irreducible, provided the characteristic of $k$ is good for 
$\fg$. For $p\!=\!3$, we have $W(1) \cong \fsl(2)$, so that $\cC_2(W(1))$ is in fact irreducible. Our example above shows that commuting varieties of Lie algebras, all whose maximal tori are 
self-centralizing, may not even be equidimensional. In contrast to $W(1)$, the ``Borel subalgebra'' $\fb \subseteq W(1)$, whose maximal tori are also self-centralizing, is an algebraic Lie algebra. 

(2) \ A consecutive application of (4) and \cite[(2.5.1),(2.5.2)]{CF} implies that the variety $\EE(2,W(1))$ of two-dimensional elementary abelian subalgebras of $W(1)$ has pure dimension $p\!-\!4$
as well as $|\Irr(\EE(2,W(1)))|$ $=\frac{p-3}{2}$. \end{Remarks}

\bigskip

\section{Algebraic Lie algebras}
Let $\fg=\Lie(G)$ be the Lie algebra of a connected algebraic group $G$. The adjoint representation 
\[ \Ad : G \lra \Aut(\fg)\]
induces an action
\[ g\dact(x,y) := (\Ad(g)(x), \Ad(g)(y))\]
of $G$ on the commuting variety $\cC_2(\fg)$ such that the surjections
\[ \pr_i : \cC_2(\fg) \lra \fg\]
are $G$-equivariant. In the sequel, we will often write $g\dact x:= \Ad(g)(x)$ for $g \in G$ and $x \in \fg$. 

Let $T \subseteq G$ be a maximal torus with character group $X(T)$, 
\[ \fg = \fg^T\!\oplus\!\bigoplus_{\alpha \in R_T} \fg_\alpha\]
be the root space decomposition of $\fg$ relative to $T$. Here $R_T \subseteq X(T)\!\smallsetminus\!\{0\}$ is the set of roots of $G$ relative to $T$, while $\fg^T := \{x \in \fg \ ; \ t\dact x = x 
\ \ \forall \ t \in T\}$ denotes the subalgebra of points of $\fg$ that are fixed by $T$. Given $x = x_0\!+\! \sum_{\alpha \in R_T}x_\alpha \in \fg$, we let 
\[ \supp(x):=\{\alpha \in R_T \ ; \ x_\alpha \ne 0\}\]
be the {\it support} of $x$. For any subset $S \subseteq X(T)$, we denote by $\ZZ S$ the subgroup of $X(T)$ generated by $S$. The group $\ZZ R_T $ is
the called the {\it root lattice} of $G$ relative to $T$.

If $H\subseteq G$ is a closed subgroup and $x \in \fg$, then $C_H(x):=\{ h \in H \ ; \ h\dact x = x\}$ is the centralizer of $x$ in $H$. 

\bigskip

\subsection{Centralizers, supports, and components}

\begin{Lem} \label{CSC1} Let $T \subseteq G$ be a maximal torus, $x \in \fg$. Then we have
\[ \dim C_T(x) = \dim T\!-\!\rk(\ZZ\supp(x)).\] \end{Lem}

\begin{proof} Writing 
\[ x = \sum_{\alpha \in R_T\cup\{0\}} x_\alpha,\]
we see that $C_T(x)=\bigcap_{\alpha \in \supp(x)}\ker\alpha = \bigcap_{\alpha \in \ZZ\supp(x)}\ker\alpha$. Since $T$ is a torus, its coordinate ring $k[T]$ is the group algebra $kX(T)$ of 
$X(T) \subseteq k[T]^\times$. By the above, the centralizer $C_T(x)$ coincides with the zero locus $Z(\{ \alpha\!-\!1 \ ; \ \alpha \in \ZZ\supp(x)\})$. Thus, letting $(k\ZZ\supp(x))^\dagger$
denote the augmentation ideal of $k\ZZ\supp(x)$, we obtain the ensuing equalities of Krull dimensions
\begin{eqnarray*}
\dim k[C_T(x)] & = & \dim k[T]/k[T]\{ \alpha\!-\!1 \ ; \ \alpha \in \ZZ\supp(x)\} =  \dim kX(T)/kX(T)(k\ZZ\supp(x))^\dagger \\
                & =& \dim k(X(T)/\ZZ\supp(x)),
\end{eqnarray*}
so that \cite[(3.2.7)]{Sp98} yields 
\[\dim C_T(x) = \dim k[C_T(x)] = \rk(X(T)/\ZZ\supp(x)) = \dim T\!-\!\rk(\ZZ\supp(x)),\] 
as desired. \end{proof}

\bigskip
\noindent
Let $\fg:= \Lie(G)$ be the Lie algebra of a connected algebraic group $G$, $\fn \subseteq \fg$ be a $G$-stable subalgebra. Then $\cC_2(\fn) \subseteq \cC_2(\fg)$ is a closed, $G$-stable subset. 
For $x \in \fn$, we define
\[ \fC(x):= \overline{G\dact(\{x\}\!\times\!C_\fn(x))} \subseteq \cC_2(\fn).\]
Then $\fC(x)=\overline{\pr_1^{-1}(G\dact x)}$ is a closed irreducible subset of $\cC_2(\fn)$ such that $\fC(x)=\fC(g\dact x)$ for all $g \in G$. 

It will be convenient to have the following three basic observations at our disposal.

\bigskip

\begin{Lem} \label{CSC2} Let $\fx,\fy : k \lra \fn$ be morphisms, $\cO \subseteq k$ be a non-empty open subset such that
\begin{enumerate}
\item[(a)] $[\fx(\alpha),\fy(\alpha)] = 0 \ \ \ \ \ \ \forall \ \alpha \in k$, and
\item[(b)] $\fx(\alpha) \in G\dact \fx(1) \  \  \  \   \  \  \  \ \forall \ \alpha \in \cO$.
\end{enumerate}
Then we have $(\fx(0),\fy(0)) \in \fC(\fx(1))$. \end{Lem}

\begin{proof} In view of (a), there is a morphism
\[ \varphi : k \lra \cC_2(\fn) \ \ ; \ \ \alpha \mapsto (\fx(\alpha),\fy(\alpha)).\]
Let $\alpha \in \cO$. Then (b) provides $g \in G$ such that $\fx(\alpha) = g\dact \fx(1)$. Thus,
\[ \varphi(\alpha) = g\dact (\fx(1),g^{-1}\dact \fy(\alpha)) \in \fC(\fx(1)) \ \ \ \ \ \ \forall \ \alpha \in \cO,\]
so that 
\[ (\fx(0),\fy(0))=\varphi(0) \in \varphi(\overline{\cO}) \subseteq \overline{\varphi(\cO)} \subseteq \fC(\fx(1)),\]
as desired. \end{proof}

\bigskip 

\begin{Lem} \label{CSC3} Let $T \subseteq G$ be a maximal torus, $x \in \fn$. Suppose that $c \in \fn\cap \fg_{\alpha_0}$ (for some $\alpha_0 \in R_T$) is such that
\begin{enumerate}
\item[(a)] $\rk(\ZZ\supp(x\!+\!c))\! >\! \rk(\ZZ\supp(x))$, and
\item[(b)] $k[c,x]=[c, C_\fn(x)]$. \end{enumerate} 
Then $\fC(x)\subseteq \fC(x\!+\!c)$. \end{Lem}

\begin{proof} Note that
\[ x\!+\! \alpha_0(t)c = t\dact (x\!+\!c) \in G\dact (x\!+\!c) \ \ \ \ \ \ \ \ \forall \ t \in C_T(x).\]
In view of Lemma \ref{CSC1}, condition (a) ensures that $\dim C_T(x\!+\!c)^\circ\!<\!\dim C_T(x)^\circ$, so that $\dim \overline{\im \alpha_0(C_T(x)^\circ)}$ $=1$. Chevalley's Theorem 
(cf.\ \cite[(I.\S8)]{Mu}) thus provides a dense open subset $\cO \subseteq k$ such that $\cO\subseteq \alpha_0(C_T(x)^\circ)$. As a result,
\[ (\ast) \ \ \ \ \ \ \ x\!+\!\lambda c \in G\dact (x\!+\!c)  \ \ \ \ \ \text{for all} \ \lambda \in \cO.\] 
Condition (b) provides a linear form $\eta \in C_\fn(x)^\ast$ such that
\[ [y,c] = \eta(y)[x,c] \ \ \ \ \ \ \ \forall \ y \in C_\fn(x).\]
Given $y \in C_\fn(x)$, we define morphisms $\fx, \fy : k \lra \fn$ via
\[ \fx(\alpha) = x\!+\!\alpha c \ \ \ \text{and} \ \ \ \fy(\alpha) := \left\{\begin{array}{cl}  y\!+\!\eta(x)^{-1}\eta(y)\alpha c & \eta(x)\!\ne\!0 \\ y & \eta(x)\!=\!0. \end{array} \right.\]
In view of ($\ast$), we may apply Lemma \ref{CSC2} to obtain
\[ (x,y) = (\fx(0),\fy(0)) \in \fC(x\!+\!c).\]
As a result, $\{x\}\!\times\!C_\fn(x) \subseteq \fC(x\!+\!c)$, whence $\fC(x) \subseteq \fC(x\!+\!c)$. \end{proof}

\bigskip

\begin{Lem} \label{CSC4} Given $x \in \fn$, let $\fv \subseteq \fn$ be a $G$-submodule such that $G\dact x \subseteq \fv$. Then the following statements hold:
\begin{enumerate}
\item If $\fC(x) \in \Irr(\cC_2(\fn))$, then $C_\fn(x) \subseteq \fv$.
\item If $C_\fn(\fv) \not \subseteq \fv$, then $\fC(x) \not \in \Irr(\cC_2(\fn))$. \end{enumerate} \end{Lem}

\begin{proof} (1) Since the component $\fC(x)$ is $\GL_2(k)$-stable, we have $C_\fn(x)\!\times\!\{x\} \subseteq \fC(x)$. Thus,
\[ C_\fn(x) \subseteq \pr_1(\fC(x)) \subseteq \overline{G\dact x} \subseteq \fv.\]
(2) Let $y \in C_\fn(\fv)\!\smallsetminus\!\fv$. Since $x \in \fv$, we have $y \in C_\fn(x)\!\smallsetminus\!\fv$, and our assertion follows from (1). \end{proof}

\bigskip

\subsection{Distinguished elements}
Let $\fg\!=\!\Lie(G)$ be the Lie algebra of a connected algebraic group $G$. In the following, we denote by $T(G)$ the maximal torus of $Z(G)$. Note that $T(G)$ is contained in any maximal 
torus $T \subseteq G$.

An element $x \in \fg$ is {\it distinguished} ({\it for $G$}), provided every torus $T \subseteq C_G(x)$ is contained in $T(G)$. If $x$ is distinguished, so is every element 
of $G\dact x$. In that case, we say that $G\dact x$ is a {\it distinguished orbit}. 

\bigskip

\begin{Lem} \label{DE1} Let $x \in \fg$. Then $x$ is distinguished if and only if $C_T(x)^\circ\!=\!T(G)$ for every maximal torus $T \subseteq G$. \end{Lem}

\begin{proof}  Suppose that $x$ is distinguished. If $T \subseteq G$ is a maximal torus, then $C_T(x)^\circ \subseteq C_G(x)$ is a torus, so that $C_T(x)^\circ \subseteq T(G)$. On the other 
hand, we have $T(G)\subseteq T$, whence $T(G) \subseteq C_T(x)^\circ$.

For the reverse direction, we let $T' \subseteq C_G(x)$ be a torus. Then there is a maximal torus $T\supseteq T'$ of $G$, so that
\[ T' \subseteq C_T(x)^\circ = T(G).\]
Hence $x$ is distinguished. \end{proof}

\bigskip

\begin{Lem} \label{DE2} Let $B \subseteq G$ be a Borel subgroup with unipotent radical $U$. We write $\fb:=\Lie(B)$ and $\fu:=\Lie(U)$. 
\begin{enumerate} 
\item If $x \in \fb$ is distinguished for $G$, then it is distinguished for $B$. 
\item If $\cO \subseteq \fg$ is a distinguished $G$-orbit, then $\cO\cap\fu$ consists of distinguished elements for $B$. \end{enumerate} \end{Lem}

\begin{proof} (1) Since $B$ is a Borel subgroup, \cite[(6.2.9)]{Sp98} yields $Z(G)^\circ = Z(B)^\circ$, whence $T(G)=T(B)$.  Let $T' \subseteq C_B(x)$ be a torus. 
Since $x$ is distinguished for $G$, we obtain $T' \subseteq T(G)=T(B)$, so that $x$ is also distinguished for $B$. 

(2) This follows directly from (1). \end{proof}

\bigskip

\begin{Lem} \label{DE3} Let $G$ be a connected algebraic group with maximal torus $T$ such that $Z(G) = \bigcap_{\alpha \in R_T}\ker \alpha$.
\begin{enumerate}
\item If $x \in \fg$ is distinguished, then $\rk(\ZZ\supp(x))=\rk (\ZZ R_T)$.
\item If $(T\cap C_G(x))^\circ$ is a maximal torus of $C_G(x)$ and $\rk(\ZZ \supp(x))=\rk (\ZZ R_T)$, then $x$ is distinguished. \end{enumerate} \end{Lem}

\begin{proof} Let $\hat{x} \in \fg$ be an element such that $\supp(\hat{x})=R_T$. By assumption, we have $Z(G)=C_T(\hat{x})$, and Lemma \ref{CSC1} implies
\[ \dim Z(G) = \dim T\!-\!\rk(\ZZ R_T).\]
By the same token, 
\[ \dim C_T(x)\!-\!\dim Z(G) = \rk(\ZZ R_T)\!-\!\rk(\ZZ \supp(x))\]
for every $x \in \fg$.

(1) Let $x \in \fg$ be distinguished. Observing $Z(G) \subseteq T$, we have $Z(G)^\circ = C_T(x)^\circ$. Hence $\rk(\ZZ R_T)\!=\!\rk(\ZZ \supp(x))$.

(2) We put $\hat{T}:= (T\cap C_G(x))^\circ$. Since $\hat{T} \subseteq C_T(x)^\circ$, we obtain $\hat{T}=C_T(x)^\circ$. Hence $\rk(\ZZ\supp(x))=\rk (\ZZ R_T)$ yields $\hat{T}=Z(G)^\circ$, so 
that $Z(G)^\circ$ is a maximal torus of $C_G(x)$. As a result, the element $x$ is distinguished. \end{proof}

\bigskip
\noindent
Recall that the semisimple rank $\ssrk(G)$ of a reductive group $G$ coincides with the rank of its derived group $(G,G)$. 

\bigskip

\begin{Cor} \label{DE4} Let $B \subseteq G$ be a Borel subgroup of a reductive group $G$, $T \subseteq B$ be a maximal torus. If $x \in \fb$ is distinguished for $B$, then
\[ \rk(\ZZ \supp(x))= \ssrk(G).\]
\end{Cor} 

\begin{proof} Let $T \subseteq B$ be a maximal torus. Then $T$ is a maximal torus for $G$ such that  $Z(G)= \bigcap_{\alpha \in R_T} \ker\alpha$, cf.\ \cite[(\S 26, Ex.4)]{Hu81}. In view of  
\cite[(6.2.9)]{Sp98}, we have 
$\dim Z(G)^\circ = \dim Z(B)^\circ$.  Lemma \ref{CSC1} implies
\begin{eqnarray*}
\dim C_T(x)\!-\!\dim Z(B) & = & \dim C_T(x)\!-\!\dim Z(G) = \rk(\ZZ R_T)\!-\!\rk(\ZZ \supp(x)) \\
                                        &= & \ssrk(G)\!-\!\rk(\ZZ\supp(x))
\end{eqnarray*}
for every $x \in \fb$, cf.\ \cite[(II.1.6)]{Ja03}. 

Let $x \in \fb$ be distinguished for $B$. Then $Z(B)^\circ \subseteq T$ is a maximal torus of $C_B(x)$ and $Z(B)^\circ \subseteq C_T(x) \subseteq C_B(x)$. Thus, $C_T(x)^\circ=Z(B)^\circ$, and
the identity above yields $\rk(\ZZ\supp(x))=\ssrk(G)$. \end{proof}

\bigskip

\subsection{Modality} \label{S:Mod}
Let $G$ be a connected algebraic group acting on an algebraic variety $X$. Given $i \in \NN_0$, we put
\[ X_{[i]} := \{x \in X \ ; \ \dim G.x = i\}.\]
Since $X_{[i]}=\emptyset$ whenever $i\!>\!\dim X$, the set $\NN_0(X):=\{ i \in \NN_0 \ ; \ X_{[i]}\ne \emptyset\}$ is finite. 

The set $X_{[i]}$ is locally closed and $G$-stable. If $x \in X_{[i]}$, then $G.x$ is closed in $X_{[i]}$.  

\bigskip
\noindent
Suppose that $G$ acts on $X$. Then
\[ \modd (G;X):= \max_{i \in \NN_0(X)} \dim X_{[i]}\!-\!i\]
is called the {\it modality of $G$ on $X$}. 

For ease of reference, we record the following well-known fact.

\bigskip

\begin{Lem} \label{Mod1}Suppose that the connected algebraic group $G$ acts on $X$. Then $\modd(G;X)=0$ if and only if $G$ acts on $X$ with finitely many orbits. 
In this case, $X_{[i]}$ has pure dimension $i$ for every $i \in \NN_0(X)$. \end{Lem}

\bigskip

\begin{Prop} \label{Mod 2}Let $G$ be a connected algebraic group with Lie algebra $\fg$ and such that $\Lie(C_G(x)) = C_\fg(x)$ for all $x \in \fg$. Then we have
\[ \dim \cC_2(\fg)= \dim G\!+\!\modd(G;\fg).\] \end{Prop}

\begin{proof} Given $x \in \fg$, the identity $\Lie(C_G(x)) = C_\fg(x)$ implies that the differential
\[ \fg \lra T_x(G\dact x) \ \ ; \ \ y \mapsto [y,x]\]
of the orbit map $g \mapsto g\dact x$ is surjective, cf. \cite[(2.2)]{Ja04}. In particular, $\rk(\ad x)= \dim G\dact x$, so that
\[ \fg_{(n)} = \fg_{[n]}.\]
Hence $\modd(\fg)=\modd(G;\fg)$, and our assertion follows from Proposition \ref{Pre3}(2). \end{proof}

\bigskip

\section{Springer Isomorphisms}
The technical condition of Proposition \ref{Mod 2} automatically holds in case $\Char(k)\!=\!0$. In this section, we are concerned with its verification for the unipotent radicals of Borel subgroups
for good characteristics of $G$. Throughout, we assume that $G$ is a connected reductive group. Following \cite[(2.6)]{Ja04} we say that the characteristic $\Char(k)$ is {\it good for $G$}, provided
$\Char(k)=0$ or the prime $p:= \Char(k)\!>\!0$ is a good prime for $G$, see loc.\ cit.\ for more details.  

\bigskip

\begin{Lemma} \label{SpI1} Let $G$ be semisimple with almost simple factors $G_1, \ldots, G_n$. For $i \in \{1,\ldots, n\}$, we let $B_i=U_i\!\rtimes\!T_i$ be a Borel 
subgroup of $G_i$ with unipotent radical $U_i$ and maximal torus $T_i$. Then the following statements hold:
\begin{enumerate}
\item $B\!:=\!B_1\cdots B_n$ is a Borel subgroup of $G$ with unipotent radical $U\!:=\!U_1\cdots U_n$ and maximal torus $T\!:=\! T_1\cdots T_n$.
\item The product morphism
\[ \mu_U : \prod_{i=1}^n U_i \lra U \ \ ; \ \ (u_1,\ldots, u_n) \mapsto u_1\cdot u_2\cdots u_n\]
is an isomorphism of algebraic groups. \end{enumerate}\end{Lemma}

\begin{proof} We consider the direct product $\hat{G} := \prod_{i=1}^n G_i$ along with the multiplication
\[  \mu_G : \hat{G} \lra G \ \ ; \ \ (g_1,\ldots, g_n) \mapsto g_1\cdot g_2\cdots g_n.\]
Since $(G_i,G_j)=e_k$ for $i\ne j$, it follows that $\mu_G$ is a surjective homomorphism of algebraic groups, cf.\ \cite[(27.5)]{Hu81} 

(1) We put $\hat{B}:= \prod_{i=1}^n B_i$, $\hat{U}:= \prod_{i=1}^n U_i$, and $\hat{T}:= \prod_{i=1}^n T_i$. These three subgroups of $\hat{G}$ are closed and connected. Moreover, they
are solvable, unipotent and diagonalizable, respectively. Direct computation shows that $\hat{U}$ is normal in $\hat{B}$ as well as $\hat{B}=\hat{U}\!\rtimes\! \hat{T}$.

Let $H\supseteq \hat{B}$ be a connected, closed solvable subgroup of $\hat{G}$. Since the $i$-th projection $\pr_i : \hat{G} \lra G_i$ is a homomorphism of algebraic groups for $1\!\le\!i\!\le\!n$, 
it follows that $H_i := \pr_i(H)\supseteq B_i$ is a closed, connected, solvable subgroup of $G_i$. Hence $H_i=B_i$, so that
\[ H \subseteq \prod_{i=1}^nH_i = \hat{B}.\]
As a result, $\hat{B}$ is a Borel subgroup of $\hat{G}$. In view of \cite[(21.3C)]{Hu81}, $B=\mu(\hat{B})$ is a Borel subgroup of $G$. Similarly, $T=\mu_G(\hat{T})$ is a maximal torus of $B$.
In addition, $B = \mu_G(\hat{B})=\mu_G(\hat{U}\!\rtimes\!\hat{T})=U\cdot T$. It follows that the unipotent closed normal subgroup $U=\mu_G(\hat{U}) \unlhd B$ is the unipotent radical of $B$. 

(2) According to \cite[(27.5)]{Hu81}, the product morphism
\[ \mu_G : \hat{G} \lra G\]
has a finite kernel. Since $\hat{G}$ is connected, it follows that $\ker \mu_G \subseteq Z(\hat{G})$, while $\hat{G}$ being semisimple forces $Z(\hat{G})$ to be diagonalizable, cf.\ \cite[(II.1.6)]{Ja03}.
As a result, the kernel $\ker \mu_U$ is diagonalizable and unipotent, so that $\ker \mu_U\!=\!\{1\}$. Since $\mu_U$ is surjective, map $\mu_U$ is a bijective morphism of algebraic varieties. 

Note that $\Lie(\hat{U})\!=\!\bigoplus_{i=1}^n\Lie(U_i)$ and that the differential $\msd(\mu_U) : \Lie(\hat{U}) \lra \Lie(U)$ is given by 
\[ (x_1,\ldots, x_n) \mapsto \sum_{i=1}^n x_i.\] 
Let $i\!\ne\!j$. Since $(T_i,U_j)\!=\!\{1\}$, we have $\Ad(t_i)|_{\Lie(U_j)}\!=\!\id_{\Lie(U_j)} \ \ \forall \ t_i \in T_i$. Thus, if $(x_1,\ldots, x_n) \in \ker\msd(\mu_U)$, then $\Ad(t)(x_i)\!=\!x_i$ for all $t \in T$ and 
$i \in \{1,\ldots, n\}$. Using the root space decomposition of $\Lie(U)$ relative to $T$, we conclude that $x_i\!=\!0$ for $i \in \{1,\ldots,n\}$. As a result, the map $\msd(\mu_U)$ is injective. Since $\mu_U$ is bijective, 
we have
\[ \dim_k\Lie(U) = \dim_k\Lie(\hat{U}) = \sum_{i=1}^n\dim_k\Lie(U_i)\]
so that $\msd(\mu_U)$ is an isomorphism. We may now apply \cite[(5.3.3)]{Sp98} to conclude that $\mu_U$ is an isomorphism as well. \end{proof}

\bigskip
\noindent
Let $B \subseteq G$ be a Borel subgroup with unipotent radical $U \unlhd B$. A $B$-equivariant isomorphism
\[ \varphi : U \lra \Lie(U)\]
will be referred to as a {\it Springer isomorphism for $B$}.

\bigskip
\noindent
Springer isomorphisms first appeared in \cite{Sp69} in the context of semisimple algebraic groups, providing a homeomorphism between the unipotent variety of a group and the nilpotent
variety of its Lie algebra. Our next result extends \cite[(2.2),(4.2)]{Go06} to the context of reductive groups.

\bigskip 

\begin{Proposition} \label{SpI2} Suppose that $\Char(k)$ is good for $G$. Let $B \subseteq G$ be a Borel subgroup with unipotent radical $U$ and put $\fu:= \Lie(U)$. 
\begin{enumerate}
\item There is a Springer isomorphism $\varphi : U \lra \fu$. 
\item We have $\Lie(C_U(x))=C_\fu(x)$ for every $x \in \fu$. \end{enumerate}\end{Proposition}

\begin{proof} (1) We first assume that $G$ is semisimple, so that $G=G_1\cdots G_n$, where $G_i \unlhd G$ is almost simple. As before, we put $\hat{G} := \prod_{i=1}^nG_i$.
Then every Borel subgroup of $\hat{G}$ is of the form $\hat{B}= \prod_{i=1}^nB_i$ for some Borel subgroups $B_i \subseteq G_i$. Hence \cite[(21.3C)]{Hu81} ensures that there exist Borel 
subgroups $B_i = U_i\!\rtimes\!T_i$ of $G_i$ such that $B=B_1\cdots B_n$ and $U=U_1\cdots U_n$. We put $\fu_i:=\Lie(U_i)$. As noted in \cite[(2.2)]{Go06}, there are Springer isomorphisms 
$\varphi_i : U_i \lra \fu_i$ for $1\!\le\!i\!\le\!n$. 

We define $\hat{B}$ and $\hat{U}$ as in the proof of Lemma \ref{SpI1} and consider the product morphisms
\[ \mu_B : \hat{B} \lra B \ \ \text{and} \ \ \mu_U : \hat{U} \lra U.\] 
Then $\Lie(\hat{U})=\bigoplus_{i=1}^n\fu_i$ and 
\[ \hat{\varphi} : \hat{U} \lra \Lie(\hat{U}) \ \ ; \ \ (u_1,\ldots,u_n) \mapsto (\varphi_1(u_1),\ldots,\varphi_n(u_n))\]
is a $\hat{B}$-equivariant isomorphism of varieties. Lemma \ref{SpI1} implies that $\mu_U: \hat{U} \lra U$ is an isomorphism of algebraic groups such that
\[ \mu_U(\hat{b}\hat{u}\hat{b}^{-1}) = \mu_B(\hat{b})\mu_U(\hat{u})\mu_B(\hat{b})^{-1} \]
for all $\hat{b} \in \hat{B}$ and $\hat{u} \in \hat{U}$. Moreover, the differential
\[\msd(\mu_U): \Lie(\hat{U}) \lra \fu\]
is an isomorphism such that
\[ \msd(\mu_U)(\Ad\hat{b}(x)) = \Ad(\mu_B(\hat{b}))(\msd (\mu_U)(x))\]
for all $\hat{b} \in \hat{B}$ and $x \in \Lie(\hat{U})$. Consequently, $\varphi := \msd(\mu_U)\circ \hat{\varphi} \circ \mu_U^{-1}$ defines an isomorphism
\[ \varphi : U \lra \fu.\]
For $b=\mu_B(\hat{b}) \in B$ and $u \in U$, we obtain, writing $b\dact x:= \Ad(b)(x)$,
\[ \varphi(bub^{-1}) = (\msd(\mu_U) \circ \hat{\varphi})(\hat{b}\mu_U^{-1}(u)\hat{b}^{-1}) = \msd(\mu_U) (\hat{b}\dact \hat{\varphi}(\mu_U^{-1}(u))) = b\dact\varphi(u),\]
as desired. 

Now let $G$ be reductive. Then $G':=(G,G)$ is semisimple, while $G=G'\cdot Z(G)^\circ$, with $Z(G)^\circ$ being a torus. Let $B \subseteq G$ be a Borel subgroup. Since 
$Z(G)^\circ \subseteq B$, we obtain $B=(B\cap G')Z(G)^\circ$, and $B$ being connected implies that $B=(B\cap G')^\circ Z(G)^\circ$. Let $B'\supseteq (B\cap G')^\circ$ be
a Borel subgroup of $G'$. Then $B'Z(G)^\circ$ is a closed, connected, solvable subgroup of $G$ containing $B$, whence $B=B'Z(G)^\circ$. As a result, $B'\subseteq B\cap G'$,
so that $B'=(B\cap G')^\circ$. 

Let $U$ be the unipotent radical of $B$. Since $Z(G)^\circ \twoheadrightarrow G/G'$ is onto, the latter group is diagonalizable, so that the canonical morphism $U \lra G/G'$ is trivial. As a result, 
$U \subseteq G'$, whence $U\subseteq (B\cap G')^\circ$. If $U'$ is the unipotent radical of $(B\cap G')^\circ$, then $B=(B\cap G')^\circ Z(G)^\circ$ implies that $U'$ is normal in $B$, whence 
$U'\subseteq U$. It follows that $U$ is the unipotent radical of the Borel subgroup $(B\cap G')^\circ$ of $G'$. The first part of the proof now provides a
$(B\cap G')^\circ$-equivariant isomorphism $\varphi : U \lra \fu$. Since $Z(G)$ acts trivially on both spaces, this map is also $B$-equivariant.

(2) In view of (1), the arguments of \cite[(4.2)]{Go06} apply. \end{proof}

\bigskip

\section{Commuting varieties of unipotent radicals}
Throughout this section, $G$ denotes a connected reductive algebraic group. If $B$ is a Borel subgroup of $G$ with unipotent radical $U$, then $B$ acts on $\fu:=\Lie(U)$ via the adjoint 
representation. Hence $B$ also acts on the commuting variety $\cC_2(\fu)$, and for every $x \in \fu$ we consider
\[ \fC(x):=\overline{B\dact (\{x\}\!\times\!C_\fu(x))}.\]
As observed earlier, we have
\[ \fC(x) = \fC(b\dact x) \ \ \ \ \ \ \ \ \  \forall \ b \in B, x \in \fu.\]

\bigskip

\subsection{The dimension formula}

\begin{Lem} \label{Df1} Let $B \subseteq G$ be a Borel subgroup with unipotent radical $U \subseteq B$, $x \in \fu:= \Lie(U)$. 
\begin{enumerate}
\item There exists a maximal torus $T \subseteq B$ such that 
\begin{enumerate}
\item $C_B(x)^\circ=C_U(x)^\circ\!\rtimes\!C_T(x)^\circ$, and
\item $\fC(x)$ is irreducible of dimension
\[ \dim \fC(x) = \dim B\!-\!\dim C_T(x)\] 
whenever $\Char(k)$ is good for $G$. \end{enumerate}
\item If $\Char(k)$ is good for $G$, then we have
\[ \dim \fC(x)=\dim B\!-\!\dim Z(G)\] 
if and only if $x$ is distinguished for $B$.
\end{enumerate} \end{Lem}

\begin{proof} (1a) Let $T' \subseteq C_B(x)^\circ$ be a maximal torus, $T\supseteq T'$ be a maximal torus of $B$. We write $B=U\!\rtimes\!T$ and recall that 
$U=B_u$ is the set of unipotent elements of $B$, see \cite[(6.3.3),(6.3.5)]{Sp98}. Thus, $C_U(x)^\circ= C_B(x)^\circ_u=B_u\cap C_B(x)^\circ$ is the unipotent radical of $C_B(x)^\circ$. 

Since $T'\subseteq C_T(x)^\circ$, while the latter group is a torus of $C_B(x)^\circ$, it follows that $T'=C_T(x)^\circ$. General theory (cf.\ \cite[(6.3.3),(6.3.5)]{Sp98}) now yields
\[ C_B(x)^\circ = C_B(x)^\circ_u\!\rtimes\!T' = C_U(x)^\circ\!\rtimes\!C_T(x)^\circ,\]
as asserted.

(1b) Since $\{x\}\!\times\!C_\fu(x)$ is irreducible, so is the closure $\fC(x)$ of its $B$-saturation. Consider the dominant morphism
\[ \omega : B\!\times\!C_\fu(x) \lra  \fC(x) \ \ ; \ \ (b,y) \mapsto (b\dact x,b\dact y).\]
We fix $(b_0\dact x,b_0\dact y_0) \in \im \omega$. Then 
\[ \zeta : C_B(x) \lra  \omega^{-1}(b_0\dact x,b_0\dact y_0) \ \ ; \ \ c \mapsto (b_0c,c^{-1}\dact y_0)\]
is a morphism with inverse morphism
\[ \eta : \omega^{-1}(b_0\dact x,b_0\dact y_0) \lra C_B(x) \ \ ; \ \ (b,y) \mapsto b_0^{-1}b.\]
As a result, $\dim\omega^{-1}(b_0\dact x,b_0\dact y_0)=\dim C_B(x)$, and the fiber dimension theorem gives
\[ \dim \fC(x) = \dim B\!+\!\dim C_\fu(x)\!-\!\dim C_B(x).\]
In view of Proposition \ref{SpI2}(2), we have $\Lie(C_U(x))=C_\fu(x)$. Consequently,
\[ \dim \fC(x) = \dim B\!+\!\dim C_U(x)^\circ\!-\!\dim C_B(x)^\circ,\]
and the assertion now follows from (1a).

(2) Suppose that $\dim \fC(x)\!=\!\dim B\!-\!\dim Z(G)$. Part (1) provides a maximal torus $T \subseteq B$ such that $\dim C_T(x)\!=\!\dim Z(G)$. This readily implies $C_T(x)^\circ\!=\!Z(G)^\circ$, so that 
$C_B(x)^\circ\!=\!Z(G)^\circ\!\ltimes\!C_U(x)^\circ$. In particular, $Z(G)^\circ$ is the unique maximal torus of $C_B(x)^\circ$, so that $x$ is distinguished for $B$. 

Suppose that $x$ is distinguished for $B$. Let $T\subseteq B$ be a maximal torus such that $C_T(x)^\circ$ is a maximal torus of $C_B(x)^\circ$. It follows that $C_T(x)^\circ\! =\!Z(G)^\circ$, whence 
$\dim \fC(x)\!=\!\dim B\!-\!\dim Z(G)$. \end{proof} 

\bigskip

\begin{Thm} \label{Df2} Suppose that $\Char(k)$ is good for $G$ and let $B \subseteq G$ be a Borel subgroup of $G$, $U \subseteq B$ be its unipotent radical, $\fu:=\Lie(U)$. Then we have
\[ \dim\cC_2(\fu)=\dim B\!-\!\dim Z(G)\!+\!\modd(B;\fu).\] \end{Thm}

\begin{proof} We first assume that $G$ is almost simple, so that $\dim Z(G)\!=\!0$. Thanks to \cite[Thm.10]{GMR}, we have
\[ \modd(U;\fu)=\modd(B;\fu)\!+\!\rk(G),\]
so that a consecutive application of Proposition \ref{SpI2} and Proposition \ref{Mod 2} implies
\[ \dim \cC_2(\fu)= \dim U\!+\!\modd(U;\fu) = \dim U\!+\!\rk(G)\!+\!\modd(B;\fu) = \dim B\!+\!\modd(B;\fu).\]
Next, we assume that $G$ is semisimple with almost simple constituents $G_1,\ldots, G_n$, say. There are Borel subgroups $B_i \subseteq G_i$ of $G_i$ with unipotent radicals 
$U_i$ such that $B\!=\!B_1\cdots B_n$ and $U\!=\!U_1\cdots U_n$. Let $\fu:=\Lie(U)$ and $\fu_i:=\Lie(U_i)$. Lemma \ref{SpI1} provides an isomorphism $U\cong \prod_{i=1}^n U_i$, 
so that $\fu = \bigoplus_{i=1}^n \fu_i$. If $x=\sum_{i=1}^nx_i \in \fu$, then $B\dact x =\prod_{i=1}^n B_i\dact x_i$, so that $\dim B\dact x = \sum_{i=1}^n \dim B_i\dact x_i$. 
This readily implies 
\[\fu_{[j]} := \{x \in \fu \ ; \ \dim B\dact x\!=\!j\} = \bigcup_{\{m \in \NN_0^n ; |m|=j\}} \prod_{i=1}^n(\fu_i)_{[m_i]} \ \ \ \ \ \ \forall \ j \in \NN_0,\] 
where we put $|m|\!:=\!\sum_{i=1}^n m_i$ for $m \in \NN_0^n$. Consequently,
\[ \dim \fu_{[j]} = \max\{ \sum_{i=1}^n \dim (\fu_i)_{[m_i]} \ ; \ m \in \NN_0^n \ , \ |m|\!=\!j \} \ \ \ \ \forall \ j \in \NN_0.\]
As a result,
\begin{eqnarray*} 
\modd(B;\fu) &=& \max_{j\ge 0} \max \{ \sum_{i=1}^n \dim (\fu_i)_{[m_i]} \ ; \ m \in \NN_0^n \ ; \ |m|\!=\!j\}\!-\!j \\
                     &= & \max_{j\ge 0} \max\{ \sum_{i=1}^n \dim (\fu_i)_{[m_i]} \!-\!m_i \ ; \ m \in \NN_0^n \ ; \ |m|\!=\!j\} \\
                     &=  & \max _{m \in \NN_0^n} \sum_{i=1}^n (\dim (\fu_i)_{[m_i]} \!-\!m_i)
                       = \sum_{i=1}^n \max_{m_i \ge 0} (\dim (\fu_i)_{[m_i]} \!-\!m_i)
                       = \sum_{i=1}^n \modd(B_i;\fu_i).
\end{eqnarray*}
Since $\cC_2(\fu) \cong \prod_{i=1}^n \cC_2(\fu_i)$, we arrive at
\[ \dim \cC_2(\fu) = \sum_{i=1}^n\dim\cC_2(\fu_i) = \sum_{i=1}^n \dim B_i\!+\!\modd(B_i;\fu_i) = \dim B\!+\!\modd(B;\fu),\]
as desired. 

If $G$ is reductive, then $G=Z(G)^\circ G'$, with $G':=(G,G)$ being semisimple and $Z(G)^\circ$ being a torus. By the arguments of Proposition \ref{SpI2}, $B':=(B\cap G')^\circ$ is a Borel subgroup 
of $G'$ with unipotent radical $U$ and such that $B=B'Z(G)^\circ$ with $Z(G)\cap B'$ being finite.  It follows that
\[ B\dact x= B'\dact x\]
for all $x \in \fu$, and the identities
\[ \dim \cC_2(\fu) = \dim B'\!+\!\modd(B';\fu)= \dim B\!-\!\dim Z(G)\!+\!\modd(B;\fu)\]
verify our claim. \end{proof}

\bigskip
\noindent
We denote by $\cO_{\rm reg} \subseteq \fg$ the regular nilpotent $G$-orbit.

\bigskip

\begin{Lem} \label{Df3} Suppose that $\Char(k)$ is good for $G$. Given $x \in \cO_\reg\cap\fu$, $\fC(x)$ is an irreducible component of $\cC_2(\fu)$ of dimension $\dim B\!-\!\dim Z(G)$. \end{Lem}

\begin{proof} By general theory,  $\cO_{\rm reg}\cap\fu$ is an open $B$-orbit of $\fu$, cf.\ \cite[(5.2.3)]{Ca}. Consequently, $\cO_{\rm reg}\cap \fu_{(\max \NN_0(\fu))}$ is a non-empty subset of $\fu$. 
Since $\fu_{(\max\NN_0(\fu))}$ is a $B$-stable subset of $\fu$, it follows that $\cO_{\rm reg}\cap\fu \subseteq \fu_{(\max \NN_0(\fu))}$. 

Let $x \in \cO_{\rm reg}\cap\fu$. Then $B\dact x \subseteq \fu_{(\max\NN_0(\fu))}$ is open in $\fu$, so that $\pr_1^{-1}(B\dact x)$ is open in $\cC_2(\fu)$. Corollary \ref{Pre4} now shows that
$\pr_1^{-1}(B\dact x)$ is an open subset of the irreducible component $\overline{\pr_1^{-1}(\fu_{(\max\NN_0(\fg))})}$ of $\cC_2(\fu)$. Consequently,
\[ \fC(x) = \overline{\pr_1^{-1}(B\dact x)} =  \overline{\pr_1^{-1}(\fu_{(\max\NN_0(\fu))})}\]
is an irreducible component of $\cC_2(\fu)$. Since the element $x$ is distinguished for $G$, Lemma \ref{DE2} shows that it is also distinguished for $B$. We may now apply Lemma \ref{Df1} to see
that $\dim \fC(x)\!=\!\dim B\!-\!\dim Z(G)$. \end{proof}

\bigskip

\begin{Remarks} (1) The foregoing result in conjunction with Theorem \ref{Df2} implies that $\cC_2(\fu)$ is equidimensional only if $B$ acts on $\fu$ with finitely many orbits.

(2) It also follows from the above and Corollary \ref{Pre4} that $\max\NN_0(\fu)\!=\!\dim_k\fu\!-\!\ssrk(G)$. \end{Remarks}

\bigskip

\subsection{Minimal supports}
As before, we let $G$ be a connected reductive algebraic group, with Borel subgroup $B\! =\! U\!\rtimes\!T$. The corresponding Lie algebras will be denoted $\fg$, $\fb$ and $\fu$. Let $R_T$ the root system of $G$ 
relative to $T$, $\Delta:=\{\alpha_1,\ldots,\alpha_n\} \subseteq R_T$ be a set of simple roots. Given $\alpha = \sum_{i=1}^nm_i\alpha_i \in R_T$, we denote by $\height(\alpha)=\sum_{i=1}^nm_i$ the {\it height} of $\alpha$ (relative to $\Delta$), and put for $x \in \fu\!\smallsetminus\!\{0\}$
\[ \deg(x):= \min\{ \height(\alpha) \ ; \ \alpha \in \supp(x)\}\]
as well as
\[ \msupp (x):= \{\alpha \in \supp(x) \ ; \ \height(\alpha)=\deg(x)\}.\]
Given $n \in \NN_0$, we put
\[ \fu^{(\ge n)} := \langle \{ x \in \fu \ ; \ \deg(x) \ge n\}\rangle.\]

\bigskip

\begin{Lem} \label{Ms1} Given $x \in \fu\!\smallsetminus\!\{0\}$, we have $\deg(b\dact x) =\deg(x)$ and $\msupp(b\dact x) = \msupp(x)$ for all $b \in B$.\end{Lem}

\begin{proof} For $u \in U$ we consider the morphism
\[ \Phi_u : U \lra U \ \ ; \ \ v \mapsto  [u,v],\]
where $[u,v]:=uvu^{-1}v^{-1}$ denotes the commutator of $u$ and $v$. According to \cite[(4.4.13)]{Sp98}, we have
\[ \msd(\Phi_u)(x) = u\dact x\!-\!x \ \ \ \ \ \ \forall \ x \in \fu.\]
Given a positive root $\alpha \in R_T^+$, we consider the root subgroup $U_\alpha$ of $U$. For $u \in U_\alpha$ and $\beta \in R^+_T$, an application \cite[(8.2.3)]{Sp98} shows that
\[ \Phi_u(U_\beta) \subseteq \prod_{i,j>0} U_{i\alpha+j\beta}.\]
Let $x \in \fu\!\smallsetminus\!\{0\}$ and put $d:= \deg(x)$. Since $\fu_\beta = \Lie(U_\beta)$, the foregoing observations in conjunction with \cite[(8.2.1)]{Sp98} yield
\[ \Ad(u)(x) \equiv x \ \ \ \ \ \modd(\fu^{(\ge d+1)})   \ \ \ \ \ \ \ \ \  \forall \ u \in U.\]
Thus, $\fu^{(\ge n)}$ is a $U$-submodule of $\fu$ for all $n\!\ge\!1$ such that $U$ acts trivially on $\fu^{(\ge n)}/\fu^{(\ge n+1)}$.  

Now write $x = \sum_{\alpha \in \msupp(x)} x_\alpha\!+\!x'$, where $x' \in \fu^{(\ge d+1)}$. Given $b \in B$, there are $t \in T$ and $u \in U$ such that $b=tu$. By the above, we obtain
\[ b\dact x \equiv \sum_{\alpha \in \msupp(x)} \alpha(t)x_\alpha \ \ \ \ \ \ \ \modd(u^{(\ge d+1)}),\]
whence $\deg(b\dact x) = \deg(x)$ and $\msupp(b\dact x) = \msupp(x)$. \end{proof}

\bigskip
\noindent
Let $\cO \subseteq \fu$ be a $B$-orbit. In view of Lemma \ref{Ms1}, we may define
\[ \msupp(\cO) := \msupp(x) \ \ \ \ \ \ (x \in \cO).\]

\bigskip

\subsection{The case $\mathbf{\modd(B;\fu)\!=\!0}$}
The case where $B$ acts on $\fu$ with finitely many orbits is governed by the Theorem of Hille-R\"ohrle \cite[(1.1)]{HR}, which takes on the following form in our context:

\bigskip 

\begin{Prop} \label{fmod1} Suppose that $\Char(k)$ is good for $G$. Then $\modd(B;\fu)=0$ if and only if every almost simple constituent of $(G,G)$ is of type
$(A_n)_{n\le4}$ or $B_2$. \end{Prop}

\begin{proof} Returning to the proof of Theorem \ref{Df2}, we let $G_1, \ldots, G_n$ be the simple constituents of $(G,G)$ and pick Borel subgroups $B_i$ of $G_i$, with unipotent
radicals $U_i$. Then 
\[ B:= Z(G)^\circ B_1\cdots B_n \]
is a Borel subgroup of $G$ with unipotent radical $U:= U_1\cdots U_n$. Setting $\fu:=\Lie(U)$ and $\fu_i:=\Lie(U_i)$, we have
\[ \modd(B;\fu) = \sum_{i=1}^n\modd(B_i;\fu_i),\]
so that \cite[(1.1)]{HR} yields the result. \end{proof} 

\bigskip

\begin{Lem} \label{fmod2} Suppose that $\modd(B;\fu)\!=\!0$. If $C \in \Irr(\cC_2(\fu))$, then there is a unique orbit $\cO_C \subseteq \pr_1(C)$ such that
\begin{enumerate}
\item[(a)] $\cO_C$ is dense and open in $\pr_1(C)$, and
\item[(b)] $C=  \fC(x)$ for all $x \in \cO_C$. \end{enumerate} \end{Lem}

\begin{proof} Since the component $C$ is $B$-stable, so is the closed subset $\pr_1(C) \subseteq \fu$, cf.\ Lemma \ref{Pre2}. By assumption, $B$ thus acts with finitely many orbits on the 
irreducible variety $\pr_1(C)$. Hence there is a $B$-orbit $\cO_C \subseteq \pr_1(C)$ such that $\overline{\cO}_C = \pr_1(C)$. Consequently, $\cO_C$ is open in $\pr_1(C)$. The unicity of 
$\cO_C$ follows from the irreducibility of $\pr_1(C)$.

Let $x \in \cO_C$, so that $\cO_C=B\dact x$. Then there is $y \in \fu$ such that $(x,y) \in C$. In particular, $y \in C_\fu(x)$, so that $(x,y) \in B\dact (\{x\}\!\times\!C_\fu(x))= \pr_1^{-1}(\cO_C)$.
Thanks to (a), $\pr^{-1}(\cO_C)$ is open in $\pr_1^{-1}(\pr_1(C))$. It follows that $(B\dact (\{x\}\!\times\!C_\fu(x)))\cap C$ is a non-empty open subset of $C$, so that
\[ C = \overline{(B\dact (\{x\}\!\times\!C_\fu(x)))\cap C} \subseteq \fC(x).\]
Since the latter set is irreducible, while $C$ is a component, we have equality. \end{proof}

\bigskip

\begin{Remarks} (1) \ The Lemma holds more generally for each $C \in \Irr(\cC_2(\fu))$ with $\modd(B;\pr_1(C))\!=\!0$.

(2) \ Suppose that $\modd(B;\fu)\!=\!0$. In view of Theorem \ref{Df2} and Lemma \ref{Df1}, each distinguished $B$-orbit $B\dact x$ gives rise to an irreducible component $\fC(x)$ of maximal dimension. \end{Remarks}

\bigskip
\noindent
Suppose that $\modd(B;\fu)\!=\!0$. Using Lemma \ref{fmod2} we define
\[ \msupp(C) = \msupp(\cO_C)\]
for every $C \in \Irr(\cC_2(\fu))$. 

\bigskip

\section{Almost simple groups}\label{S:Asg}
The purpose of this technical section section is the proof of the following result, which extends \cite[\S3]{GR} to good characteristics. 

\bigskip

\begin{Proposition} \label{Asg1} The following statements hold:
\begin{enumerate}
\item If $G$ has type $(A_n)_{n\le 4}$, then $\cC_2(\fu)$ is equidimensional and 
\[ |\Irr(\cC_2(\fu))| = \left\{ \begin{array}{cc} 5 & n\!=\!4 \\ 2 & n\!=\!3 \\ 1 & \text{else.} \end{array} \right.\]
\item If $\Char(k)\!\ne\!2$ and $G$ has type $B_2$, then $\cC_2(\fu)$ is equidimensional and $|\Irr(\cC_2(\fu))|=2$. \end{enumerate} \end{Proposition}

\bigskip
\noindent
For $G$ as above, the Borel subgroup $B \subseteq G$ acts on $\fu$ with finitely many orbits. We let $\fR \subseteq \fu$ be a set of orbit representatives, so that
\[ \cC_2(\fu) = \bigcup_{ x \in \fR} \fC(x)\]
is a finite union of closed irreducible subsets. We will determine in each case the set $\{ x \in \fR \ ; \ \fC(x) \in \Irr(\cC_2(\fu))\}$. A list of orbit representatives is given in \cite[\S 3]{GR} and
we will follow the notation established there.

\bigskip

\subsection{Special linear groups}\label{S:SL}
Let $G=\SL_{n+1}(k)$ and $\fg = \fsl_{n+1}(k)$, where $1\!\le\!n\!\le\!4$. Moreover, $B,T,U$ denote the standard subgroups of upper triangular, diagonal, and upper unitriangular
matrices, respectively. 

For $i\!\le\!j \in \{1,\ldots, n\!+\!1\}$, we let $E_{i,j}$ be the $(i,j)$-elementary matrix, so that 
\[ \fu:= \bigoplus_{i<j}kE_{i,j}\]
is the Lie algebra of the unipotent radical $U$ of $B$. We denote the set of simple roots by  $\Delta :=\{\alpha_1,\ldots,\alpha_n\}$. 
Let $i\!<\!j\!\le\! n\!+\!1$. Then $E_{i,j}$ is the root vector corresponding to the root $\alpha_{i,j}:=\sum_{\ell=i}^{j-1}\alpha_\ell$. We therefore have $\alpha_i=\alpha_{i,i+1}$
for $1\!\le\!i\!\le \!n$  and
\[ R_T^+ := \{ \alpha_{i,j} \ ; \ 1\!\le\!i\!<\!j\!\le\!n\!+\!1\}\]
is the set of roots of $\fu$ relative to $T$. (The set of positive roots of $\fsl_{n+1}(k)$.) 

Recall that 
\[ E_{i,j}E_{r,s} = \delta_{j,r}E_{i,s},\]
as well as
\[ [E_{i,j},E_{r,s}] = \delta_{j,r}E_{i,s}\!-\!\delta_{s,i}E_{r,j} \ \ \ \ \ \ \text{for all} \ i,j,r,s \in \{1,\ldots,n\!+\!1\}. \]
Let $\alpha = \alpha_{i,j}$ be a positive root. Then
\[ U_\alpha := \{ 1\!+\! aE_{i,j} \ ; \ a \in k\}\]
is the corresponding root subgroup of $U$, and the formula above implies
\[ \Ad(1\!+\! aE_{i,j})(x) = (1\!+\! aE_{i,j})x(1\!-\! aE_{i,j}) = x\!+\!a[E_{i,j},x]\]
for all $x \in \fu$. 

Note that $A:=\{(a_{ij}) \in \Mat_{n+1}(k) \ ; \ a_{ij}=0 \ \text{for} \ i\!>\!j\}$ is a subalgebra of the associative algebra $\Mat_{n+1}(k)$. We consider the linear map
\[ \zeta : A \lra A \ \ ; \ \ E_{i,j} \mapsto E_{n+2-j,n+2-i}.\]
Then we have 
\begin{enumerate}
\item[(a)] $\zeta(ab) = \zeta(b)\zeta(a)$ for all $a,b \in A$, and 
\item[(b)] $\det(\zeta(a)) = \det(a)$ for all $a \in A$. \end{enumerate}
There results a homomorphism
\[ \tau : B \lra B \ \ ; \ \ a \mapsto \zeta(a)^{-1}\]
of algebraic groups such that $\tau(U)=U$. We write $\fb:=\Lie(B)$ and put $\Upsilon := \msd(\tau)|_\fu$. As $\zeta$ is linear, \cite[(4.4.12)]{Sp98} implies that
\[\Upsilon(E_{i,j}) = -E_{n+2-j,n+2-i} \ \ \ \ \ \ 1\!\le\!i\!<\!j\!\le\!n\!+\!1.\]
Thus, $\Upsilon$ is an automorphism of $\fu$ of order $2$ such that
\[ \Upsilon(\fu_{\alpha_{ij}})= \fu_{\alpha_{n+2-j,n+2-i}}.\]
Since $\Delta$ is a basis for the root lattice $\ZZ R_T^+ = \ZZ R_T$, there is an automorphism  $\sigma : \ZZ R_T^+ \lra \ZZ R_T^+$ of order $2$ such that
\[\sigma(\alpha_i) = \alpha_{n+1-i} \ \ \ \ \ \ \ 1\!\le\!i\!\le\!n.\]
Thus, $\sigma(R_T^+) = R_T^+$, and 
\[ \Upsilon(\fu_\alpha)= \fu_{\sigma(\alpha)} \ \ \ \ \ \ \forall \ \alpha \in R_T^+.\]
We denote by $(\fu^n)_{n \in \NN}$ the descending series of the nilpotent Lie algebra $\fu$, which is inductively defined via $\fu^1:=\fu$ and $\fu^{n+1} := [\fu,\fu^n]$.
Note that $\fu^n = \fu^{(\ge n)}$ for all $n\!\ge\!1$.

\bigskip

\begin{Lem} \label{SL1} Let $C \in \Irr(\cC_2(\fu))$. Then we have
\[ \msupp([\Upsilon\!\times\!\Upsilon](C))=\sigma(\msupp(C)).\]
\end{Lem}

\begin{proof} We put $\cO_C = B\dact x$. In view of  $\Upsilon = \msd(\tau)|_\fu$, we have
\[ \Upsilon(b\dact x) = \tau(b)\dact \Upsilon(x) \  \  \  \  \   \  \  \forall \ b \in B, x \in \fu.\]
Consequently,
\[ \Upsilon(\cO_C) = \Upsilon(B\dact x) = B\dact \Upsilon(x)\]
is an open orbit of $\Upsilon(\pr_1(C)) = \pr_1([\Upsilon\!\times\!\Upsilon](C))$, so that
\[ \cO_{[\Upsilon\!\times\!\Upsilon](C)} = \Upsilon(\cO_C).\]
Setting $d\!:=\!\deg(x)$, we have
\[ x \equiv \sum_{\alpha \in \msupp(x)} x_\alpha \ \ \ \ \ \ \modd \fu^{(\ge d+1)}.\]
Thus, 
\[ \Upsilon(x) \equiv  \sum_{\alpha \in \msupp(x)} -x_{\sigma(\alpha)}  \ \ \ \ \ \modd \fu^{(\ge d+1)},\]
whence
\[\msupp([\Upsilon\!\times\!\Upsilon](C)) = \msupp(\Upsilon(x)) = \sigma(\msupp(x)) = \sigma(\msupp(C)),\]
as desired. \end{proof} 

\bigskip

\begin{Remark} The list of orbit representatives for the case $A_4$ given in \cite[(3.4)]{GR} contains some typographical errors, which we correct
as follows:
\begin{enumerate} 
\item[(a)] In the form stated loc.\ cit., the element $e_3$ satisfies $\rk(\ZZ\supp(e_3))\!=\!3$, so that it is not distinguished, see Corollary \ref{DE4}.
We write $e_3 = 11010{\bf 1} 0000$, so that $e_3 = \Upsilon(e_7)$.  
\item[(b)] In \cite[(3.4)]{GR}, we have $e_4\!=\!e_5$. We put $e_4:=1101000000$ (the element, $e_3$ of \cite[(3.4)]{GR}), so that $e_4=\Upsilon(e_8)$.
\end{enumerate} \end{Remark} 

\bigskip

\begin{Lem} \label{SL2} Suppose that $\Char(k)\!\ne\!2$. Let $G\!=\!\SL_5(k)$. Then $\cC_2(\fu)$ is equidimensional and $|\Irr(\cC_2(\fu))|\!=\!5$. \end{Lem}

\begin{proof} Let $C \in \Irr(\cC_2(\fu))$ be a component and pick $x \in \cO_C$, so that $C=\fC(x)$, cf.\ Lemma \ref{fmod2}. We consider
\[ S_C := \msupp(C)\cup\msupp([\Upsilon\!\times\!\Upsilon](C)) = \msupp(x)\cup\msupp(\Upsilon(x)).\]
According to Lemma \ref{SL1}, $S_C$ is a $\sigma$-stable subset of $R_T^+$.

We will repeatedly apply Lemma \ref{CSC4} to $B$-submodules of $\fu$.

\medskip

(a) {\it We have $x \not\in  \bigcup_{i=1}^3 kE_{i,i+2}\!\oplus\!\fu^3$}.

\smallskip
\noindent
Suppose that $x \in kE_{i,i+2}\!\oplus\!\fu^3$ for some $i \in \{1,2,3\}$. Since $\fu^3\!=\!kE_{1,4}\!\oplus\!kE_{2,5}\!\oplus\!kE_{1,5}$, we have $[E_{2,3},\fu^3]\!=\!(0)$. It thus follows from 
Lemma \ref{CSC4} that $\deg(x)\!\le\!2$. Consequently,  $\deg(x)\!=\!2$ and $|\msupp(x)|\!=\!1$. If $|S_C|\!=\!1$, then $i\!=\!2$. Since $[E_{2,3},kE_{2,4}\!+\!\fu^3]\!=\!(0)$, we may 
apply Lemma \ref{CSC4} to $\fv\!:=\!kE_{2,4}\!+\!\fu^3$ to obtain a contradiction. Alternatively, we may assume that $i\!=\!1$. As $[E_{2,4},kE_{1,3}\!+\!\fu^3]\!=\!(0)$, another application
of Lemma \ref{CSC4} rules out this case. \hfill $\diamond$ 

\medskip

(b) {\it We have $\deg(x)\!=\!1$ and $|S_C|\!=\! 2,4$}. 

\smallskip
\noindent
Suppose that $\deg(x)\!\ge\!2$. In view of (a), we have $\deg(x)\!=\!2$ and $|\msupp(x)|\!\ge \! 2$. If $|\msupp(x)|\!=\!2=\!|S_C|$, then $\msupp(x)\!=\!S_C$ is $\sigma$-stable, so that $\msupp(x) =
\{\alpha_{1,3},\alpha_{3,5}\}$. Thus, $B\dact x \subseteq \fv:= kE_{1,3}\!\oplus\!kE_{3,5}\!\oplus\!\fu^3$ (see also Lemma \ref{Ms1}). Since $E_{2,4} \in C_\fu(\fv)$, Lemma \ref{CSC4} yields 
a contradiction. If $|\msupp(x)| \!= \!2$ and $|S_C|\!=\!3$, then $\msupp(x)\cap\msupp(\Upsilon(x))$ contains a fixed point of $\sigma$, and we may assume that $\msupp(x) = \{\alpha_{1,3},
\alpha_{2,4}\}$. In view of \cite[(3.4)]{GR}, we may assume that $x= e_{48}=E_{1,3}\!+\!E_{2,4}$. Since $B\dact x \subseteq \fv:= kE_{1,3}\!\oplus\!kE_{2,4}\!\oplus\!\fu^3$, while $E_{1,2}\!+\!E_{3,4} \in C_\fu(x)$,
Lemma \ref{CSC4} yields a contradiction.

We thus assume that $|\msupp(x)|\!=\!3$. Then \cite[(3.4)]{GR} in conjunction with Lemma \ref{Ms1} gives $x\!=\!e_{47}\!=\!E_{1,3}\!+\!E_{2,4}\!+\!E_{3,5}$. Since $E_{1,2}\!+\!E_{3,4} \in C_\fu(x)$, while 
$B\dact x \subseteq \fu^2$, this contradicts Lemma \ref{CSC4}. 

Consequently, $\deg(x)\!=\!1$, so that $\msupp(x) \subseteq \Delta$. Since $\sigma$ acts without fixed points on $\Delta$, every $\sigma$-orbit of $\Delta$ has two elements. As $S_C \subseteq \Delta$
 is a disjoint union of $\sigma$-orbits, we obtain $|S_C|\!=\!2,4$. \hfill $\diamond$
 
\medskip
(c) {\it We have $|\msupp(x)|\!\ge\!2$}.

\smallskip
\noindent
Alternatively, (b) provides $i \in \{1,\ldots,4\}$ such that $B\dact x \subseteq \fv:= kE_{i,i+1}\!+\!\fu^2$. Applying $\Upsilon$, if necessary, we may assume that $i \in \{1,2\}$. 

Suppose that $i\!=\!1$. Then Lemma \ref{Ms1} in conjunction with \cite[(3.4)]{GR} implies that we have to consider the following cases:
\begin{eqnarray*} 
 & x = e_{16} = E_{1,2}\!+\!E_{2,4}\!+\!E_{3,5} \ ; \ x=e_{17} = E_{1,2}\!+\!E_{2,4} \ ; \ x=e_{18} = E_{1,2}\!+\!E_{3,5}\!+\!E_{2,5} \ ; \\
 & x=e_{19} = E_{1,2}\!+\!E_{3,5} \ ; \  x  = e_{20} = E_{1,2}\!+\!E_{2,5} \ ; \ x=e_{21} = E_{1,2}.
 \end{eqnarray*}
Consequently, $E_{3,4} \in C_\fu(x)\!\smallsetminus\!\fv$, which contradicts Lemma \ref{CSC4}.

Suppose that $i\!=\!2$. Then \cite[(3.4)]{GR} implies 
\begin{eqnarray*} 
 & x = e_{29} = E_{2,3}\!+\!E_{3,5}\!+\!E_{1,4} \ ; \ x=e_{30} = E_{2,3}\!+\!E_{3,5} \ ; \ x=e_{31} = E_{2,3}\!+\!E_{1,4} \ ; \\
 & x=e_{32} = E_{2,3}\!+\!E_{1,5} \ ; \  x=e_{33} = E_{2,3}.
 \end{eqnarray*}
Since $E_{4,5} \in [C_\fu(e_{30})\cap C_\fu(e_{32})\cap C_u(e_{33})]\!\smallsetminus\!\fv$, Lemma \ref{CSC4} rules out these possibilities. In view of $E_{4,5}\!+\!E_{1,3} \in 
C_\fu(e_{29})\!\smallsetminus\!\fv$, it remains to discuss the case where $x\!=\!e_{31}$.

We consider the morphism
\[ \fx : k \lra \fu \ \ ; \ \ \alpha \mapsto e_{31}\!+\!\alpha E_{3,5}.\]
Then we have $\fx(\alpha) \in B\dact e_{29}$ for all $\alpha \in k^\times$, while $\fx(0)=e_{31}$. Direct computation shows that
\[ C_\fu(e_{31}) = kE_{2,3}\!\oplus\!kE_{1,3}\!\!\oplus\!kE_{2,4}\!\oplus\!\fu^3.\]
For $y = a E_{2,3}\!+\!bE_{1,3}\!+\!cE_{2,4}\!+\!z \in C_\fu(e_{31})$, where $z \in \fu^3$, we consider the morphism
\[ \fy : k \lra \fu \ \ ; \ \ \alpha \mapsto y\!+\!b\alpha E_{4,5}\!+\!a\alpha E_{3,5}.\]
Since $[\fx(\alpha),\fy(\alpha)]\!=\!0$ for all $\alpha \in k^\times$, Lemma \ref{CSC2} yields
\[ (e_{31},y) = (\fx(0),\fy(0)) \in \fC(\fx(1))=\fC(e_{29}).\]
Consequently, $\fC(e_{31}) \subseteq \fC(e_{29})$. Since $\fC(e_{29}) \not \in \Irr(\cC_2(\fu))$, we again arrive at a contradiction. \hfill $\diamond$ 

\medskip
(d) { \it We have $|S_C|\!=\!4$}.

\smallskip
\noindent
Suppose that $|S_C|\!\ne\!4$. Then (b) implies $|S_C|\!=\!2$ and (c) shows that $\msupp(x) \subseteq \Delta$ is $\sigma$-stable with $2$ elements. Consequently, $\msupp(x)=\{\alpha_1,\alpha_4\}$, 
or $\msupp(x)=\{\alpha_2,\alpha_3\}$. 

If $x\!=\!E_{2,3}\!+\!E_{3,4}\!+\!y$, where $y \in \fu^2$, then \cite[(3.4)]{GR} yields $x \in B\dact e_{23}\cup B\dact e_{24}$, where $e_{23} := E_{2,3}\!+\!E_{3,4}\!+\!E_{1,5}$ and
$e_{24} := E_{2,3}\!+\!E_{3,4}$. We may invoke Lemma \ref{CSC3} to see that $\fC(e_{24}) \subseteq \fC(e_{23})$. It was shown in \cite[(3.4)]{GR}, that $\fC(e_{23})\subseteq \fC(e_1)$. 
Hence $\fC(x)$ is not a component, a contradiction.

It follows that $\msupp(x)=\{\alpha_1,\alpha_4\}$, so that \cite[(3.4)]{GR} implies 
\[x \in B\dact e_{13} \cup B\dact e_{14} \cup B\dact e_{15},\]
where $e_{15}:=E_{1,2}\!+\!E_{4,5}$, $e_{14} := e_{15}+E_{2,5}$ and $e_{13} := e_{15}\!+\!E_{2,4}$. In view of $C_\fu(e_{15}) \subseteq kE_{1,2}\!\oplus\!kE_{4,5}\!\oplus\!\fu^2$, we have
$[C_\fu(e_{15}),E_{2,5}] \subseteq k[E_{1,2},E_{2,5}] =k[e_{15},E_{2,5}]$. Lemma \ref{CSC3} thus shows that $\fC(e_{15}) \subseteq \fC(e_{14})$. In \cite[(3.4)]{GR} it is shown that
$\fC(e_{14}) \subseteq \fC(e_3)$. According to (b), the latter set is not a component, so that $\fC(e_{14})$ isn't either. 

It remains to dispose of the case $x\!=\!e_{13}$. For $(\alpha, \beta) \in k^2$, we consider the elements
\[ e_1(\alpha,\beta) := E_{1,2}\!+\!\alpha E_{2,3}\!+\!\beta E_{3,4}\!+\!E_{4,5} \ \ \text{and} \ \ e_{13}(\alpha,\beta) := e_1(\alpha,\beta)\!+\!E_{2,4}\]
of $\fu$. Let $u_{i,j}(t):=1\!+\!tE_{i,j} \in U  \ \ (t \in k)$, so that $u_{i,j}(t)\dact x = x\!+\!t[E_{i,j},x]$ for all $x \in \fu$. We thus obtain
$e_{13}(\alpha,\beta) = u_{2,3}(\beta^{-1})u_{1,2}(\alpha^{-1}\beta^{-1})\dact e_1(\alpha,\beta)$ for $\alpha\beta \ne 0$. As a result,
\[ e_{13}(\alpha,\beta) \in B\dact e_1 \ \ \ \ \text{for} \ \alpha\beta \ne 0,\]
where $e_1\!=\!e_1(1,1)$. 

Direct computation shows that
\[ C_\fu(e_{13}) = ke_{13}\!\oplus\!kE_{1,3}\!\oplus\!kE_{3,5}\!\oplus\!k(E_{1,4}\!+\!E_{2,5})\!\oplus\!kE_{1,5}.\]
Let $y= a e_{13}\!+\!bE_{1,3}\!+\!c E_{3,5}\!+\!d(E_{1,4}\!+\!E_{2,5})\!+\!e E_{1,5} \in C_\fu(e_{13})$ be such that $b,c \ne 0$. We consider the morphisms
\[ \fx : k \lra \fu \ \ ; \ \ \alpha \mapsto e_{13}(\alpha,\alpha cb^{-1}) \ \ \ \text{and} \ \ \ \fy : k \lra \fu \ \ ; \ \ \alpha \mapsto y\!+\!\alpha a E_{2,3}\!+\!\alpha acb^{-1}E_{3,4}\!+\!\alpha c E_{2,4}\]
and observe that
\begin{enumerate}
\item[(a)] $\fx(\alpha) \in B\dact \fx(1)$ for all $\alpha \in k^\times$, and
\item[(b)] $[\fx(\alpha),\fy(\alpha)] = 0$ for all $\alpha \in k$. \end{enumerate}
Thus, Lemma \ref{CSC2} implies that $(e_{13},y) = (\fx(0),\fy(0)) \in \fC(\fx(1))=\fC(e_1)$. Since the set of those $y$ with $bc\ne 0$ lies dense in $C_\fu(e_{13})$, it follows that $\fC(e_{13}) \subseteq
\fC(e_1)$, a contradiction.  \hfill $\diamond$

\medskip
\noindent
If $\msupp(x)\!=\!S_C$, (d) shows that $\deg(x)\!=\!1$ and $|\msupp(x)|\!=\!4$. Hence $x$ is regular and $\fC(x)\!=\!\fC(e_1)$ is an irreducible component.

If $|\msupp(x)|\! =\! 2$, then $S_C = \msupp(x) \sqcup \sigma(\msupp(x))$ and we only need to consider the cases 
\[ \msupp(x)=\{\alpha_1,\alpha_2\} \ ; \  \{\alpha_1,\alpha_3\}. \]
If $\msupp(x)=\{\alpha_1,\alpha_2\}$, then Lemma \ref{Ms1} yields $B\dact x \subseteq \fv:=kE_{1,2}\!+\!kE_{2,3}\!+\!\fu^2$, while \cite[(3.4)]{GR}
implies 
\[ x= e_5 = E_{1,2}\!+\!E_{2,3}\!+\!E_{3,5} \ ; \ x = e_6 = E_{1,2}\!+\!E_{2,3}.\]
Consequently, $E_{4,5} \in C_\fu(x)\!\smallsetminus\!\fv$, a contradiction.

If $\msupp(x)=\{\alpha_1,\alpha_3\}$, then $B\dact x \subseteq \fv:= kE_{1,2}\!+\!kE_{3,4}\!+\fu^2$ and \cite[(3.4)]{GR} implies 
\begin{eqnarray*}  x= e_9 = E_{1,2}\!+\!E_{3,4}\!+\!E_{2,4}\!+\!E_{2,5} \ ; \  x= e_{10} = E_{1,2}\!+\!E_{3,4}\!+\!E_{2,4} \ ; \\ 
 x= e_{11} = E_{1,2}\!+\!E_{3,4}\!+\!E_{2,5}  \ ; \  x= e_{12} = E_{1,2}\!+\!E_{3,4}.
\end{eqnarray*}
Given $(\alpha,\beta) \in k^2$, we put
\[  x(\alpha,\beta) = E_{1,2}\!+\!E_{3,4}\!+\!\alpha E_{2,4}\!+\! \beta E_{2,5}.\]
Note that
\[ x(\alpha,\beta) \in B\dact x(1,1)=B\dact e_9 \ \ \ \ \ \ \text{for} \ \alpha, \beta \ne 0.\]
We put $\fw:= kE_{3,4}\!\oplus\!k(E_{1,3}\!+\!E_{2,4})\!\oplus\!kE_{3,5}\!\oplus\!kE_{1,4}\!\oplus\!kE_{1,5}$. Direct computation shows that
\[ C_\fu(x(\alpha,\beta)) = k(E_{1,2}\!+\!\alpha E_{2,4}\!+\!\beta E_{2,5})\!\oplus\! \fw\]
for all $(\alpha,\beta) \in k^2$. We have $e_i = x(\delta_{i,10},\delta_{i,11})$ for $i \in \{10,11,12\}$. Thus, if $y = a(E_{1,2}\!+\!\delta_{i,10}E_{2,4}\!+\!\delta_{i,11}E_{2,5})\!+\!w \in C_\fu(e_i)$,
where $a \in k$ and $w \in \fw$, then 
\[ y(\alpha,\beta) = y\!+\!(a\alpha\!-\!a)\delta_{i,10}E_{2,4}+\!(a\beta\!-\!a)\delta_{i,11}E_{2,5} \in C_\fu(x(\alpha,\beta)).\]
Let $i \in \{10,11,12\}$. Then the morphisms
\[ \fx_i : k \lra \fu \ \ ; \ \ \alpha \mapsto x(\alpha(\delta_{i,11}\!+\!\delta_{i,12})\!+\!\delta_{i,10}, \alpha(\delta_{i,10}\!+\!\delta_{i,12})\!+\!\delta_{i,11})\]
and 
\[ \fy_i : k \lra \fu \ \ ; \  \   \alpha \mapsto y(\alpha(\delta_{i,11}\!+\!\delta_{i,12})\!+\!\delta_{i,10}, \alpha(\delta_{i,10}\!+\!\delta_{i,12})\!+\!\delta_{i,11})\]
Fulfill the conditions of Lemma \ref{CSC2}, so that
\[ (e_i, y) = (\fx_i(0),\fy_i(0)) \in \fC(\fx_i(1)) = \fC(e_9).\]
As a result, $\fC(e_i) \subseteq \fC(e_9)$ for $10\!\le\!i\!\le\!12$. 

We have $\dim_k \im (\ad e_9)(\fb)\!=\!\dim_k \im (\ad e_9)\!+\!4$, so that $C_\fu(e_9)\!=\!C_\fb(e_9)$. Thus, Proposition \ref{SpI2} implies 
\[ \dim C_B(e_9) \le \dim_k C_\fb(e_9) = \dim_k C_\fu(e_9) = \dim C_U(e_9),\]
so that $C_B(e_9)^\circ \!= \! C_U(e_9)^\circ$. Consequently, the element $e_9$ is distinguished and $\fC(e_9)$ is a component. Hence $\Upsilon(e_9)$ is also distinguished and \cite[(3.4)]{GR} in 
conjunction with Corollary \ref{DE4} implies that $\fC(e_{25})$ is also a component. 

It remains to consider the case, where $|\msupp(x)|\!=\!3$. Then $\msupp(x)\cap\sigma(\msupp(x))$ is a $\sigma$-stable subset of $\Delta$ of cardinality $2$, so that
\[\msupp(x)\cap\sigma(\msupp(x)) = \{\alpha_1,\alpha_4\} \ ; \ \{\alpha_2,\alpha_3\}.\]
Suppose that $\msupp(x)\cap\sigma(\msupp(x)) = \{\alpha_1,\alpha_4\}$. Then we may assume that $\msupp(x) = \{\alpha_1,\alpha_2,\alpha_4\}$. Thanks to \cite[(3.4)]{GR}
this yields $x=e_3, e_4$. The above methods show that $\fC(e_4) \subseteq \fC(e_3)$, while $e_3$ is a distinguished element. Hence $\fC(e_3)$ and $\Upsilon(\fC(e_3))=\fC(e_7)$ are components 
of $\cC_2(\fu)$.

We finally consider  $\msupp(x)\cap\sigma(\msupp(x)) = \{\alpha_2,\alpha_3\}$ and assume that $\msupp(x)=\{\alpha_1,\alpha_2,\alpha_3\}$. By \cite[(3.4)]{GR}, this implies
\[x=e_2 = E_{1,2}\!+\!E_{2,3}\!+\!E_{3,4}.\]
As $\fC(e_2) \subseteq \fC(e_1)$, this case yields no additional components. It follows that
\[ \Irr(\cC_2(\fu))=\{\fC(e_1),\fC(e_3),\fC(e_7),\fC(e_9),\fC(e_{25})\},\]
so that $|\Irr(\cC_2(\fu))|\!=\!5$. \end{proof}

\bigskip

\begin{Lem} \label{SL3} Let $G\!=\!\SL_4(k)$. Then $\cC_2(\fu)$ is equidimensional with $|\Irr(\cC_2(\fu))|\!=\!2$. \end{Lem}

\begin{proof} We consider $\GL_n(k)\!=\! \SL_n(k)Z(\GL_n)$ along with its standard Borel subgroup $B_n\!=\! U_n\!\rtimes\!T_n$ of upper triangular matrices, where
$U_n$ and $T_n$ are the groups unitriangular and diagonal matrices, respectively. The $B$ orbits of $\fu_n\!:=\!\Lie(U_n)$ coincide with those of the standard Borel subgroup
$B_n\cap \SL_n(k)$ of $\SL_n(k)$. 

We consider $G'\!:=\! \GL_5(k)$ along with its commuting variety $\cC_2(\fu')$. In view of Lemma \ref{SL2},  we have
\[ \Irr(\cC_2(\fu'))=\{\fC(e'_1),\fC(e'_3),\fC(e'_7),\fC(e'_9),\fC(e'_{25})\}.\]
Let $A'$ and $A$ be the associative algebras of upper triangular $(5\!\times\!5)$-matrices and upper triangular $(4\!\times\!4)$-matrices, respectively. Then
\[ \pi : A' \lra A \ \ ; \ \ (a_{ij}) \mapsto (a_{ij})_{1\le i \le j \le 4}\]
is homomorphisms of $k$-algebras. Thus, if we identify $G:=\GL_4(k)$ with a subgroup of the Levi subgroup of $G'$, given by $\Delta_4:= \{\alpha'_1,\alpha'_2,\alpha'_3\}$, then 
the restriction
\[ \pi : B' \lra B\]
is a homomorphism of groups such that $\pi|_B\! =\! \id_B$. It follows that the differential
\[ \msd(\pi) : \fu' \lra \fu\]
of the restriction $\pi|_{U'} : U' \lra U$ is split surjective such that
\[ \msd(\pi)(b'\dact x') = \pi(b')\dact \msd(\pi)(x') \ \ \ \ \ \ \ \text{for all} \ b' \in B', x' \in \fu'.\] 
As a result, the morphism
\[ [\msd(\pi)\!\times\!\msd(\pi)] : \cC_2(\fu') \lra \cC_2(\fu) \]
is surjective and such that
\[ [\msd(\pi)\!\times\!\msd(\pi)](B'\dact (\{x'\}\!\times\!C_{\fu'}(x')) \subseteq  B\dact (\{\msd(\pi)(x')\}\!\times\! C_{\fu}(\msd(\pi)(x'))),\] 
whence
\[  [\msd(\pi)\!\times\!\msd(\pi)](\fC(x')) \subseteq \fC(\msd(\pi)(x')) \ \ \ \ \ \ \ \ \text{for all} \ x' \in \fu'.\]
Consequently,
\[\Irr(\cC_2(\fu)) \subseteq \{\fC(\msd(\pi)(e'_1)),\fC(\msd(\pi)(e'_3)),\fC(\msd(\pi)(e'_7)),\fC(\msd(\pi)(e'_9)),\fC(\msd(\pi)(e'_{25}))\}.\]
Thanks to \cite[(3.3),(3.4)]{GR}, we obtain
\[ \msd(\pi)(e'_1)=e_1 \ ; \ \msd(\pi)(e'_3) \in B\dact e_2 \ ; \ \msd(\pi)(e'_7) = e_3 \ ; \ \msd(\pi)(e'_9) = e_3 \ ; \ \msd(\pi)(e'_{25}) = e_8.\]
In \cite[(3.3)]{GR}, the authors show that $\fC(e_8) \subseteq \fC(e_1)$. By applying Lemma \ref{CSC2} to the morphism
\[ \fx : k \lra \fu \ \ ; \ \ \alpha \mapsto E_{1,2}\!+\!E_{2,3}\!+\!\alpha E_{3,4}\]
we obtain $\fC(e_2) \subseteq \fC(e_1)$.

Since the element $e_1$ is regular, it is distinguished. As $\dim_k (\ad e_3)(\fb) = \dim_k(\ad e_3)(\fu)\!+\!3=5$, we obtain, observing Proposition \ref{SpI2},
\[ \dim C_B(e_3) \le \dim_kC_\fb(e_3)=\dim_k C_\fu(e_3)= \dim C_U(e_3),\] 
so that $C_B(e_3)^\circ\! =\! C_U(e_3)^\circ$. Hence $e_3$ is distinguished for $B$, and $\Irr(\cC_2(\fu))\! =\!\{\fC(e_1),\fC(e_3)\}$. \end{proof}

\bigskip
\noindent
The same method readily shows:

\bigskip

\begin{Lem} \label{SL4} Let $G\!=\!\SL_n(k)$, where $n\!=\!2,3$. Then $\cC_2(\fu)$ is irreducible. \end{Lem}

\bigskip

\subsection{Symplectic groups}
The following result disposes of the remaining case:

\bigskip

\begin{Lem} \label{Sp1} Suppose that $\Char(k)\!\ne\!2$. Let $G\!=\!\Sp(4)$ be of type $B_2\!=\!C_2$. Then $\fC_2(\fu)$ is equi-dimensional with $|\Irr(\cC_2(\fu))|\!=\!2$. \end{Lem}

\begin{proof} Recall that $R_T^+:=\{\alpha, \beta, \alpha\!+\!\beta, \alpha\!+\!2\beta\}$ is a system of positive roots, where $\Delta=\{\alpha,\beta\}$. Suppose that $\fC(x)$ is a component.
Since $[\fu_\alpha, \fu^{(\ge 2)}]=(0)$, Lemma \ref{CSC4} implies $\deg(x)\!=\!1$. 

Suppose that $|\msupp(x)|\!=\!1$. If $\msupp(x)\!=\!\{\alpha\}$, then \cite[(3.5)]{GR} yields $x \in B\dact x_\alpha\! \cup B\dact (x_\alpha\!+\!x_{\alpha+2\beta})$, while Lemma \ref{CSC3} gives 
$\fC(x_\alpha) \subseteq \fC(x_\alpha\!+\!x_{\alpha+2\beta})$. 

Alternatively, $x \in B\dact x_\beta$. Since $C_\fu(x_\beta)= kx_\beta\!\oplus\!kx_{\alpha+2\beta}$, we have $[x_\alpha, C_\fu(x_\beta)] = k [x_\alpha, x_\beta]$ and Lemma \ref{CSC3} implies 
$\fC(x_\beta) \subseteq \fC(x_\alpha\!+\!x_\beta)$. As a result,
\[\cC_2(\fu)=\fC(x_{\alpha}\!+\!x_{\beta})\cup\fC(x_\alpha\!+\!x_{\alpha+2\beta}).\] 
Since $\Char(k)\!\ne\!2$, the arguments of Lemma \ref{SL3} show that these elements are distinguished. Consequently, $\Irr(\cC_2(\fu))=\{\fC(x_{\alpha}\!+\!x_{\beta}),
\fC(x_\alpha\!+\!x_{\alpha+2\beta})\}$. \end{proof} 

\bigskip

\subsection{Proof of Proposition \ref{Asg1}}
\begin{proof} (1) Let us first consider an almost simple group $G$ of type $A_n$ for $n \in \{1,\ldots, 4\}$. In view of \cite[(II.1.13),(II.1.14)]{Ja03}, there is a covering $\pi :\SL_{n+1}(k) \lra G$. Hence
$\pi$ is surjective and $\ker \pi \subseteq Z(G)$ is diagonalizable. Let $B_{n+1} \subseteq \SL_{n+1}(k)$ be a Borel subgroup, $U_{n+1} \unlhd B_{n+1}$ be its unipotent radical with Lie algebra
$\fu_{n+1}$. Then $B\!:=\! \pi(B_{n+1})$ is a Borel subgroup of $G$ with unipotent radical $U\!:=\!\pi(U_{n+1})$. Since $\ker \pi\cap U_{n+1} = \{1\}$, it follows that $\pi|_{U_{n+1}}$ is a closed embedding, 
so that $\pi|_{U_{n+1}} : U_{n+1} \lra U$ is an isomorphism. Consequently, its differential 
\[ \msd(\pi) : \fu_{n+1} \lra \fu\]
is an isomorphism of Lie algebras such that 
\[\pi(b)\dact \msd(\pi)(x) = \msd(\pi)(b\dact x) \ \ \ \ \ \ \ \ \ \forall \ x \in \fu_{n+1}, b \in B_{n+1}.\]
Thanks to Section \ref{S:SL}, the variety $\cC_2(\fu_{n+1})\!\cong\!\cC_2(\fu)$ is equidimensional with $|\Irr(\cC_2(\fu))|\!=\!|\Irr(\cC_2(\fu_{n+1}))|$.

(2) Since $\Sp(4)$ is simply connected, we may use the foregoing arguments in conjunction with Lemma \ref{Sp1}. \end{proof}

\bigskip

\subsection{Irreducibility and equidimensionality of $\cC_2(\fu)$}
We record the following direct consequence of Proposition \ref{Asg1}:

\bigskip

\begin{Cor} \label{Irr1} Let $G$ be connected, reductive such that $\Char(k)$ is good for $G$. Suppose that $B \subseteq G$ is a Borel subgroup with unipotent radical $U$, whose Lie algebra
is denoted $\fu$.
\begin{enumerate}
\item If $\modd(B;\fu)\!=\!0$, then $\cC_2(\fu)$ is equidimensional.
\item $\cC_2(\fu)$ is irreducible if and only if every almost simple component of $(G,G)$ is of type $A_1$ or $A_2$. \end{enumerate} \end{Cor}

\begin{proof} Let $G_1, \ldots, G_n$ be the almost simple components of $G$. As before, we may write 
\[ B=Z(G)^\circ B_1\cdots B_n,\]
where $B_i \subseteq G_i$ is a Borel subgroup. Letting $U_i$ be the unipotent radical of $B_i$ and setting $\fu_i := \Lie(U_i)$, we have $\cC_2(\fu) \cong \prod_{i=1}^n\cC_2(\fu_i)$.
This shows that
\[ \Irr(\cC_2(\fu)) = \{ \prod_{i=1}^n C_i \ ; \ C_i \in \Irr(\cC_2(\fu_i)) \ \ 1\!\le\!i\!\le\!n\}.\]
(1) The Theorem of Hille-R\"ohrle shows that each $G_i$ is of type $(A_n)_{n\le 4}$ or $B_2$. Thanks to Proposition \ref{Asg1}, each $\cC_2(\fu_i)$ is equidimensional. Hence $\cC_2(\fu)$
enjoys the same property. 

(2) If $\cC_2(\fu)$ is irreducible, then so is each $\cC_2(\fu_i)$, and a consecutive application of Theorem \ref{Df2}, Lemma \ref{Df3}, \cite[(1.1)]{HR} and Proposition \ref{Asg1} ensures that each 
almost simple group $G_i$ is of type $A_1$ or $A_2$.  The reverse direction follows directly from Proposition \ref{Asg1}. \end{proof}

\bigskip

\begin{Remark} Suppose that $G$ is almost simple of type $A-D$. If $p\!\ge\!h(G)$ is good for $G$, then \cite[(1.7),(1.8)]{SFB} in conjunction with the foregoing result implies that the
variety $V(U_2)$ of infinitesimal one-paramenter subgroups of the second Frobenius kernel $U_2$ of $U$ is irreducible if and only if $G$ is of type $A_1$ or $A_2$. \end{Remark}  

\bigskip

\section{The variety $\A(2,\fu)$}
Let $\fu\!:=\!\Lie(U)$ be the Lie algebra of the unipotent radical $U$ of a Borel subgroup $B$ of a connected reductive group $G$. In this section, we are interested in the projective variety
\[ \A(2,\fu) := \{ \fa \in \Gr_2(\fu) \ ; \ [\fa,\fa]=(0)\}\]
of two-dimensional abelian subalgebras of $\fu$. Recall that
\[ \cO_2(\fu):=\{(x,y) \in \cC_2(\fu) \ ; \ \dim_kkx\!+\!ky=2\}\]
is an open, $\GL_2(k)$-stable subset of $\cC_2(\fu)$, while the map
\[ \varphi : \cO_2(\fu) \lra \A(2,\fu) \ \ ; \ \ (x,y) \mapsto kx\!+\!ky\]
is a surjective morphism such that $\varphi^{-1}(\varphi(x,y)) = \GL_2(k)\dact (x,y)$ for all $(x,y) \in \cO_2(\fu)$. Note that $\GL_2(k)$ acts simply on $\cO_2(\fu)$, so that each fiber of $\varphi$
is $4$-dimensional.

The Borel subgroup $B$ acts on $\A(2,\fu)$ via
\[ b\dact \fa := \Ad(b)(\fa) \ \ \ \ \ \ \ \ \forall \ b \in B, \fa \in \A(2,\fu).\]
Moreover, the set $\cO_2(\fu)$ is $B$-stable and $\varphi : \cO_2(\fu) \lra \A(2,\fu)$ is $B$-equivariant. 

\bigskip

\begin{Lemma} \label{Var1} Suppose that $\ssrk(G)\!\ge\!2$. Then the following statements hold:
\begin{enumerate}
\item Given $x \in \fu\!\smallsetminus\!\{0\}$, there is $y \in \fu$ such that $(x,y) \in \cO_2(\fu)$.
\item $\cO_2(\fu)$ lies dense in $\cC_2(\fu)$. \end{enumerate} \end{Lemma}

\begin{proof} (1) Let $z \in C(\fu)\!\smallsetminus\!\{0\}$. If $x \in \fu\!\smallsetminus\!kz$, then $(x,z) \in \cO_2(\fg)$. Alternatively, $x \in kz\!\smallsetminus\!\{0\}$. Since $\ssrk(G)\!\ge\!2$, we have 
$\dim_k\fu\!>\!1$, so that there is $y \in \fu\!\smallsetminus\!kx$. It follows that $(x,y) \in \cO_2(\fu)$.  

(2) Let $x \in \fu\!\smallsetminus\!\{0\}$. By (1), there is $y \in \fu$ such that $(x,y) \in \cO_2(\fu)$. Given $\beta \in k$, we consider the morphism
\[ f_\beta : k \lra \cC_2(\fu) \ \ ; \ \ \alpha \mapsto (x,\beta x\!+\!\alpha y).\]
Then we have $f_\beta(k^\times) \subseteq \cO_2(\fu)$, so that $f(k) \subseteq \overline{\cO_2(\fu)}$. In particular, $(x,\beta x) = f(0) \in  \overline{\cO_2(\fu)}$. Setting
$\beta=0$, we obtain $(x,0) \in  \overline{\cO_2(\fu)}$. Using the $\GL_2(k)$-action, we conclude that $(0,x) \in  \overline{\cO_2(\fu)}$. Since 
\[ g : k \lra  \cC_2(\fu) \ \ ; \ \ \alpha \mapsto (\alpha x,0)\]
is a morphism such that $g(k^\times) \subseteq \overline{\cO_2(\fu)}$, we conclude that $(0,0) \in \overline{\cO_2(\fu)}$. As a result, $\cC_2(\fu)\!=\!\overline{\cO_2(\fu)}$. \end{proof}

\bigskip

\begin{Lemma} \label{Var2} Suppose that $\Char(k)$ is good for $G$ and that $\ssrk(G)\!\ge\!2$. Let $\cO \subseteq \fu\!\smallsetminus\!\{0\}$ be a $B$-orbit, $x \in \cO$.
\begin{enumerate}
\item We have $\varphi(B\dact(\{x\}\!\times\!C_\fu(x))\cap \cO_2(\fu))\!=\!\{\fa \in \A(2,\fu) \ ; \ \fa\cap \cO\ne \emptyset\}$.
\item If $\cO\!=\!\cO_{\rm reg}\cap\fu$, then $\overline{\varphi(B\dact(\{x\}\!\times\!C_\fu(x))\cap \cO_2(\fu))}$ is an irreducible component of $\A(2,\fu)$ of dimension $\dim B\!-\!\dim Z(G)\!-\!4$.
\end{enumerate}
\end{Lemma}

\begin{proof} (1) We put $\A(2,\fu)_\cO := \{\fa \in \A(2,\fu) \ ; \ \fa\cap \cO\ne \emptyset\}$. Let $y \in C_\fu(x)$ be such that $(x,y) \in \cO_2(\fu)$.Then $x \in \varphi(x,y)\cap \cO$, 
so that $\varphi(x,y) \in \A(2,\fu)_\cO$. Since $\A(2,\fu)_\cO$ is $B$-stable, it follows that $\varphi(B\dact(\{x\}\!\times\!C_\fu(x))\cap \cO_2(\fu)) = B\dact \varphi((\{x\}\!\times\!C_\fu(x))\cap \cO_2(\fu)) 
\subseteq \A(2,\fu)_\cO$.

Now suppose that $\fa \in \A(2,\fu)_\cO$, and write $\fa = ky\!\oplus\!kz$, where $y \in \cO$. Then there is $b \in B$ such that $x\!=\!b\dact y$, so that $b\dact\fa \in \varphi((\{x\}\!\times\!C_\fu(x))\cap
\cO_2(\fu))$. As a result, $\fa \in \varphi(B\dact(\{x\}\!\times\!C_\fu(x))\cap \cO_2(\fu))$.

(2) General theory tells us that $\cO\!=\!\cO_{\rm reg}\cap \fu$ is an open $B$-orbit of $\fu$. Note that $\cO_{\rm reg}$ is a conical subset of $\fg$, so that $\cO_{\rm reg}\cap\fu$ is a conical subset of
$\fu$. It now follows from (1) and \cite[(3.2)]{CF} that $\varphi(B\dact(\{x\}\!\times\!C_\fu(x))\cap \cO_2(\fu))$ is an open subset of $\A(2,\fu)$. In view of Lemma \ref{Var1}, the irreducible set 
$\{x\}\!\times\!C_\fu(x)$ meets $\cO_2(\fu)$, so that $B\dact((\{x\}\!\times\!C_\fu(x))\cap \cO_2(\fu))=B\dact(\{x\}\!\times\!C_\fu(x))\cap \cO_2(\fu)$ is irreducible. Hence
$\varphi(B\dact(\{x\}\!\times\!C_\fu(x))\cap \cO_2(\fu))$ is a non-empty irreducible, open subset of $\A(2,\fu)$. Let $C\supseteq \varphi(B\dact(\{x\}\!\times\!C_\fu(x))\cap \cO_2(\fu))$ be an irreducible
component of $\A(2,\fu)$. Then $\varphi(B\dact(\{x\}\!\times\!C_\fu(x))\cap \cO_2(\fu))$ lies dense in $C$, so that $C=\overline{\varphi(B\dact(\{x\}\!\times\!C_\fu(x))\cap\cO_2(\fu))}$. Observing Lemma 
\ref{Df1}, we thus obtain 
\[ \dim C = \dim B\dact(\{x\}\!\times\!C_\fu(x))\cap\cO_2(\fu)\!-\!4 = \dim B\dact(\{x\}\!\times\!C_\fu(x))\!-\!4=\dim B \!-\!\dim Z(G)\!-\!4,\] 
as desired. \end{proof}

\bigskip
\noindent
Given $x \in \fu$, we put
\[ \A(2,\fu,x):=\{ \fa \in \A(2,\fu) \ ; \ x \in \fa\}.\]

\bigskip

\begin{Proposition} \label{Var3} Suppose that $\Char(k)$ is good for $G$ and that $\ssrk(G)\!\ge\!2$. 
\begin{enumerate}
\item  $\dim \A(2,\fu) \!=\! \dim B\!-\!\dim Z(G)\!+\!\modd(B;\fu)\!-\!4$.
\item The variety $\A(2,\fu)$ is equidimensional if and only if every almost simple component of $(G,G)$ is of type $(A_n)_{n\le 4}$ or $B_2$. In that case, every irreducible component 
$C \in \Irr(\A(2,\fu))$ is of the form $C=\overline{B\dact\A(2,\fu,x)}$ for some $B$-distinguished element $x \in \fu$.
\item The  variety $\A(2,\fu)$ is irreducible if and only if every almost simple component of $(G,G)$ is of type $A_1$ or $A_2$. \end{enumerate} \end{Proposition}

\begin{proof} (1) We write
\[ \cC_2(\fu) = \bigcup_{C \in \Irr(\cC_2(\fu))}C\]
as the union of its irreducible components. Since $\ssrk(G)\!\ge\!2$, Lemma \ref{Var1} shows that $\cO_2(\fu)$ is a dense open subset of $\cC_2(\fu)$. As a result, every irreducible component 
$C \in \Irr(\cC_2(\fu))$ meets $\cO_2(\fu)$. In view of Theorem \ref{Df2}, we obtain 
\[ \dim \cO_2(\fu) = \dim \cC_2(\fu) = \dim B\!-\!\dim Z(G)\!+\!\modd(B;\fu).\] 
Let $C\in \Irr(\cC_2(\fu))$.  Then $C\cap\cO_2(\fu)$ is a $\GL_2(k)$-stable, irreducible variety of dimension $\dim C$, so that 
\[\dim \overline{\varphi(C\cap\cO_2(\fu))} = \dim C\cap\cO_2(\fu)\!-\!4 = \dim C\!-\!4.\]
Consequently,
\[ \dim \A(2,\fu) = \max_{C \in \Irr(\cC_2(\fu))} \overline{\varphi(C\cap\cO_2(\fu))}= \dim \cC_2(\fu)\!-\!4 = \dim B\!-\!\dim Z(G)\!+\!\modd(B;\fu)\!-\!4.\]
(2) Suppose that $\A(2,\fu)$ is equidimensional. As Lemma \ref{Var2} provides $C \in \Irr(\A(2,\fu))$ such that $\dim C\!=\!\dim B\!-\!\dim Z(G)\!-\!4$, it follows from (1) that $\modd(B;\fu)\!=\!0$. The 
Theorem of Hille-R\"ohrle (see Proposition \ref{fmod1}) ensures that every almost simple component of $(G,G)$ is of the asserted type. Assuming this to be the case, Corollary \ref{Irr1} implies that 
$\cC_2(\fu)$ is equidimensional. In view of \cite[(2.5.1)]{CF}, $\cO_2(\fu)$ is equidimensional as well. We may thus apply \cite[(2.5.2)]{CF} to the canonical surjection $\cO_2(\fu) \twoheadrightarrow
\A(2,\fu)$ and the $\GL_2(k)$-action on $\cO_2(\fu)$ to conclude that $\A(2,\fu)$ is equidimensional. 

Given $C \in \Irr(\cC_2(\fu))$, Lemma \ref{fmod2} provides $x_C \in \fu$ such that $C\! =\! \fC(x_C)$. In view of Lemma \ref{Df1}, our current assumption shows that $x_C$ is distinguished for $B$.
According to Lemma \ref{Var1}, we have $(\{x_C\}\!\times\!C_\fu(x_C))\cap\cO_2(\fu) \ne \emptyset$, while Lemma \ref{Var2} yields $\varphi(B\dact(\{x_C\}\!\times\!C_\fu(x_C))\cap\cO_2(\fu))=
B\dact\A(2,\fu,x_C)$. 

Let $a \in \fC(x_C)\cap\cO_2(\fu)$. If $\cU \subseteq \cC_2(\fu)$ is an open subset containing $a$, then $\cU\cap (B\dact(\{x_C\}\!\times\!C_\fu(x_C))$ is a non-empty open subset of the irreducible set
$B\dact(\{x_C\}\!\times\!C_\fu(x_C))$. Since this also holds for $B\dact (\{x_C\}\!\times\!C_\fu(x_C))\cap\cO_2(\fu)$, we conclude that $\cU\cap B\dact (\{x_C\}\!\times\!C_\fu(x_C))\cap \cO_2(\fu) \ne 
\emptyset$. This shows that $a \in \overline{B\dact (\{x_C\}\!\times\!C_\fu(x_C))\cap\cO_2(\fu)}$. Consequently,
\begin{eqnarray*} 
\A(2,\fu) & = &\bigcup_{C \in \Irr(\cC_2(\fu))} \varphi(\fC(x_C)\cap\cO_2(\fu)) \subseteq \bigcup_{C \in \Irr(\cC_2(\fu))} \varphi(\overline{B\dact (\{x_C\}\!\times\!C_\fu(x_C))\cap\cO_2(\fu)}) \\
                & \subseteq & \bigcup_{C \in \Irr(\cC_2(\fu))} \overline{\varphi(B\dact (\{x_C\}\!\times\!C_\fu(x_C))\cap\cO_2(\fu))}  \subseteq \bigcup_{C \in \Irr(\cC_2(\fu))} \overline{\varphi(B\dact 
                [(\{x_C\}\!\times\!C_\fu(x_C))\cap\cO_2(\fu)])} \\ 
                & = & \bigcup_{C \in \Irr(\cC_2(\fu))} \overline{B\dact \varphi((\{x_C\}\!\times\!C_\fu(x_C))\cap\cO_2(\fu))} = \bigcup_{C \in \Irr(\cC_2(\fu))} \overline{B\dact \A(2,\fu,x_C)} \subseteq \A(2,\fu),
\end{eqnarray*}
so that $\A(2,\fu)=\bigcup_{C \in \Irr(\cC_2(\fu))} \overline{B\dact \A(2,\fu,x_C)}$ is a finite union of closed irreducible subsets. It follows that every irreducible component of $\A(2,\fu)$ is of the form 
$\overline{B\dact \A(2,\fu,x_C)}$ for some $C \in \Irr(\cC_2(\fu))$. 

(3) Suppose that $\A(2,\fu)$ is irreducible. Then (2), Proposition \ref{fmod1}, and Corollary \ref{Irr1} show that the variety $\cC_2(\fu)$ is equidimensional. Using \cite[(2.5.2)]{CF}, we conclude that 
$\cC_2(\fu)$ is irreducible and Corollary \ref{Irr1} implies that $G$ has the asserted type. The reverse implication is a direct consequence of Corollary \ref{Irr1}.  \end{proof}

\bigskip

\begin{Remark} The arguments of (2) can actually be used to show that $\cC_2(\fu)$ and $\A(2,\fu)$ have the same number of components in case one (and hence both) of these spaces is (are) 
equidimensional: Let $C \in \Irr(\cC_2(\fu))$. Returning to the proof of Proposition \ref{Pre3}(3), we find a subset $X_C \subseteq \fu$ such that
\[ C= \overline{\pr_1^{-1}(X_C)}.\]
Since $C$ is $\GL_2(k)$-stable, we conclude that $X_C\not \subseteq\{0\}$. Let $x \in X_C\!\smallsetminus\!\{0\}$. Then $\{x\}\!\times\!C_\fu(x) \subseteq C$. The assumption $C_\fu(x)\!=\!kx$
implies $x \in C(\fu)$ and hence $\dim_k\fu\!=\!1$, a contradiction. As a result, $C\cap\cO_2(\fu)\!\ne\!\emptyset$. In view of \cite[(2.5.1)]{CF}, the variety $\cO_2(\fu)$ is therefore equidimensional with 
$|\Irr(\cO_2(\fu))|\!=\!|\Irr(\cC_2(\fu))|$. By virtue of \cite[(2.5.2)]{CF}, we obtain $|\Irr(\cO_2(\fu))|\!=\!|\Irr(\A(2,\fu))|$. \end{Remark}


\bigskip

\bigskip

\begin{thebibliography}{00}
\bibitem{Ca} R. Carter, \textit{Finite Groups of Lie Type}. Wiley-Interscience, New York 1993.
\bibitem{CF} H. Chang and R. Farnsteiner, \textit{Varieties of elementary abelian Lie algebras and degrees of modules}. arxiv: 1707.02580 [math. RT]
\bibitem{Fa04} R. Farnsteiner, \textit{Varieties of tori and Cartan subalgebras of restricted Lie algebras}. Trans. Amer. Math. Soc. \textbf{356} (2004), 4181--4236.
\bibitem{Go06} S. Goodwin, \textit{On the conjugacy classes in maximal unipotent subgroups of simple algebraic groups}. Transform. Groups \textbf{11} (2006), 51--76.
\bibitem{GMR} S. Goodwin, P. Mosch and G. R\"ohrle, \textit{On the coadjoint orbits of maximal unipotent subgroups of reductive groups}. Transform. Groups \textbf{21} (2016), 399-426.
\bibitem{GR} S. Goodwin and G. R\"ohrle, \textit{On commuting varieties of nilradicals of Borel subalgebras of reductive Lie algebras}. Proc. Edinb. Math. Soc. \textbf{58} (2015), 169--181.
\bibitem{HR} L. Hille and G. R\"ohrle, \textit{A classification of parabolic subgroups of classical groups with a finite number of orbits on the unipotent radical}. Transform. Groups \textbf{4} (1999), 35--52. 
\bibitem{Hu81} J. Humphreys, \textit{Linear Algebraic Groups.} Graduate Texts in Mathematics \textbf{21}. Springer-Verlag, 1981.
\bibitem{Ja03} J. Jantzen, \textit{Representations of Algebraic Groups}. Mathematical Surveys and Monographs \textbf{107}. Amer. Math. Soc. Providence, RI, 2003.
\bibitem{Ja04} \bysame, \textit{Nilpotent orbits in representation theory}. Progress in Mathematics \textbf{228}, 1--211. Birkh\"auser Boston, Boston, MA, 2004.
\bibitem{Le} P. Levy, \textit{Commuting varieties of Lie algebras over fields of prime characteristic}. J. Algebra \textbf{250} (2002), 473--484.
\bibitem{Mu} D. Mumford, \textit{The Red Book of Variety and Schemes}. Lecture Notes in Mathematics \textbf{1358}. Springer-Verlag, 1999.
\bibitem{Pr0} A. Premet, \textit{The theorem on restriction of invariants and nilpotent elements in $W_n$}, Math. USSR Sb. \textbf{73} (1992), 135--159.
\bibitem{Pr} \bysame, \textit{Nilpotent commuting varieties of reductive Lie algebras}. Invent. math. \textbf{154} (2003), 653--683.
\bibitem{Sh} I. Shafarevich, \textit{Basic Algebraic Geometry 2}. Springer-Verlag, 1997.
\bibitem{Sp69} T. Springer, \textit{The unipotent variety of a semisimple group}. Algebraic Geometry (papers presented at the Bombay Colloquium, 1968) 373--391, Tata Institute, 1969. 
\bibitem{Sp98} \bysame, \textit{Linear Algebraic Groups.} Progress in Mathematics \textbf{9}. Birkh\"auser-Boston, 1998  (2nd edition).
\bibitem{SF} H. Strade and R. Farnsteiner, \textit{Modular Lie Algebras and their Representations}. Pure and Applied Mathematics \textbf{116}. Marcel Dekker, New York, 1988.
\bibitem{SFB} A. Suslin, E. Friedlander and C. Bendel, \textit{Infinitesimal $1$-paramenter subgroups and cohomology}. J. Amer. Math. Soc. \textbf{10} (1997), 693--728.
\bibitem{YC} Y.-F. Yao and H. Chang, \textit{The nilpotent commuting variety of the Witt algebra}. J. Pure Appl. Algebra \textbf{218} (2014), 1783--1791.
\end{thebibliography}
\end{document}